\documentclass[10pt,a4paper]{amsart}
\usepackage[utf8]{inputenc}
\usepackage[T1]{fontenc}
\usepackage{amsmath}
\usepackage{amssymb}
\usepackage{graphicx}
\usepackage{hyperref}
\usepackage{cleveref}
\usepackage{tikz-cd}
\usepackage[margin=3.5cm]{geometry}

\newtheorem*{thmintro}{Theorem}
\newtheorem*{corintro}{Corollary}
\newtheorem*{propintro}{Proposition}

\theoremstyle{definition}
\newtheorem{theorem}[subsection]{Theorem}
\newtheorem{definition}[subsection]{Definition}
\newtheorem{proposition}[subsection]{Proposition}
\newtheorem{corollary}[subsection]{Corollary}
\newtheorem{lemma}[subsection]{Lemma}

\newtheorem{remark}[subsection]{Remark}            
\newtheorem{example}[subsection]{Example}

\newtheorem*{conjecture*}{Conjecture}

\RequirePackage{amssymb, relsize, geometry, indentfirst, t1enc, amsfonts, amsmath, amscd, epic, latexsym, verbatim, xspace, setspace, bm}
\RequirePackage{tikz-cd, graphicx,graphics}
\usepackage{pifont}
\newcommand{\R}{\mathbb{R}}
\newcommand{\Z}{\mathbb{Z}}
\newcommand{\C}{\mathbb{C}}
\newcommand{\J}{\mathcal{J}}
\newcommand{\Cl}{\mathbb{C}l}
\newcommand{\Ja}{J_{\textrm{alg}}}
\newcommand{\Jd}{J_{\textrm{der}}}
\newcommand{\Lt}{\mathcal{L}}
\newcommand{\Lta}{\overline{\mathcal{L}}}
\newcommand{\Ht}{\mathcal{H}}

\newcommand{\Dt}{\mathfrak{d}_t}
\newcommand{\Dct}{\overline{\Dt}}
\newcommand{\D}{\mathfrak{B}}
\newcommand{\Dch}{\mathfrak{d}}
\newcommand{\levic}{\widetilde{\nabla}}
\newcommand{\mub}{\overline{\mu}}
\newcommand{\del}{\partial}
\newcommand{\delb}{\overline{\partial}}
\newcommand{\cD}{\overline{\D}}
\newcommand{\rD}{\widetilde{D}}
\newcommand{\e}{\epsilon}
\newcommand{\ba}{\overline}{}

\newcommand{\ex}{\textrm{exp}}

\newcommand{\trD}{\mathfrak{D}}
\newcommand{\om}{\omega_0}

\newcommand{\tD}{\widetilde{\rotatebox[origin=c]{180}{D}}}

\newcommand{\inter}{\lrcorner \ }
\newcommand{\M}{\mathcal{M}}
\newcommand{\deltab}{\ba{\delta}}
\newcommand{\har}{\boldsymbol{\mathcal{H}}}
\newcommand{\bp}{\mathcal{P}}
\newcommand{\tr}{\mathcal{Q}}
\newcommand{\rhod}{\rho_{\del}}
\newcommand{\rhodb}{\ba{\rhod}}

\newcommand{\taub}{\ba{\tau}}
\newcommand{\taud}{\tau_{\del}}
\newcommand{\taudb}{\ba{\taud}}
\newcommand{\DA}{D_{A}}
\newcommand{\hd}{\varepsilon}
\newcommand{\hdb}{\overline{\hd}}
\newcommand{\DEA}{d_{A}}
\newcommand{\bD}{B}
\newcommand{\hst}{\text{\ding{73}}}
\newcommand{\trs}{\hst}
\newcommand{\A}{a}

\title{More Differential Operators on Almost Hermitian Manifolds}

\begin{document}
    \author{Samuel Hosmer}
	\address{New York City Department of Health and Mental Hygiene}
	\email{samuel.hosmer@gmail.com}
	\subjclass[2020]{}
	
	\begin{abstract}
	In 1980 Michelsohn defined a differential operator on sections of the complex Clifford bundle over a compact K{\"a}hler manifold $M$. This operator is a differential and its Laplacian agrees with the Laplacian of the Dolbeault operator on forms through a natural identification of differential forms with sections of the Clifford bundle. Relaxing the condition that $M$ be K{\"a}hler, we introduce two differential operators on sections of the complex Clifford bundle over a compact almost Hermitian manifold which naturally generalize the one introduced by Michelsohn. We show surprising K{\"a}hler-like symmetries of the kernel of the Laplacians of these operators in the almost Hermitian and almost K{\"a}hler settings, along with a correspondence of these operators to operators on forms which are of present interest in almost complex geometry.
	\end{abstract}

\maketitle

\setcounter{tocdepth}{1}
\tableofcontents

\section{Introduction}

Given a smooth manifold $M$ de Rham gave us a way of viewing the differential forms on $M$ as a cochain complex computing the real cohomology of $M$ with differential $d$, the exterior derivative on forms. The space of forms $\Omega(M)$ is integrally graded as the space of sections of the exterior algebra bundle on the dual of $TM$. Additionally $d$ is an elliptic differential operator of degree 1. This graded differential algebra is a fundamental object of study within geometry and topology, with Sullivan in 1960 showing that $\big(\Omega(M), d\big)$ holds the real homotopy type of $M$. 

In the presence of a metric on $M$ the operator $d$ and sections of the exterior bundle $\Omega(M) = \Gamma(\Lambda(M))$ can be replaced with the Riemannian Dirac operator $\rD$ and sections of the Clifford bundle $\Gamma(Cl(M))$, respectively. The bundle $Cl(M)$ is isomorphic to $\Lambda(M)$ as a vector bundle, but the multiplication is not preserved by the isomorphism. The operator $\rD$ is an elliptic self adjoint operator on sections of $Cl(M)$  which factorizes the Laplacian of $d$. That is, $\rD = d + d^*.$

There are some notable consequences of this replacement: $Cl(M)$ has no natural integral grading and the Dirac operator is generically not a differential. In effect, by this switch, we lose the cochain complex. What we gain is Hodge theory.

Recent advances in almost complex geometry show that a valuable topological insight into an almost complex manifold $M$ is given by the harmonic forms of certain differential operators on $\Omega(M)$. On an almost complex manifold $M$ of dimension $2n$ the exterior derivative $d$ over the complexification $\Omega_\C(M)$ decomposes into 4 components in bidegree. Namely,

$$ d = \del + \mub + \delb + \mu $$ where $\del, \mub, \delb, \mu$ are of bidegree $(1,0), (-1,2), (0,1), (2,-1)$ respectively.

Given a compatible metric there is a real linear isomorphism induced from complex conjugation $c$ and a complex linear isomorphism induced by the Hodge star operator $\ast$ so that

$$ c: \Omega_\C(M)^{p,q} \xrightarrow{\sim} \Omega_\C(M)^{q,p} \hspace{.3in} $$
$$ \ast: \Omega_\C(M)^{p,q} \xrightarrow{\sim}\Omega_\C(M)^{n-q,n-p} .$$ 

 Letting $\omega = \langle J \cdot, \cdot \rangle$, there is a classical $sl(2)$ representation on the bundle $\Omega(M)$ generated by the operators $\big( L,\Lambda, H \big)$ where $L = \omega \wedge \cdot, \Lambda = L^*, $ and $ H = [\Lambda, L]$. In the case the form $\omega$ is $d$ closed, there are local commutation relations between the components of the exterior derivative and these operators $L,\Lambda, H$ called the almost K{\"a}hler identities, which translate over compact manifolds via Hodge theory to a representation on these classical sl(2) operators on the space of certain harmonic forms. For $T$ an elliptic differential operator on forms or sections of the Clifford bundle we denote by $\har_{T} = \ker(\Delta_T)$ the space of $T$ harmonic forms.  

Cirici and Wilson show in \cite{Ci_Wi} 
that over a compact almost K{\"a}hler manifold $M$ the operators $(L, \Lambda, H)$ define a finite dimensional $sl(2)$ representation on 
$$\bigoplus_{p,q \geq 0}\har_\del^{p,q} \cap \har_{\mub}^{p,q} = \bigoplus_{p,q \geq 0} \har_d^{p,q} $$ where $\har_T^{p,q} $ are the $T$ harmonic forms in a particular bidegree  $(p,q)$.  

Tardini and Tomassini in \cite{Ta_To} consider the operator $\delta = \del + \mub$ and showed the corresponding $\Z$ graded result that  $(L, \Lambda, H)$ define an $sl(2)$ representation on $$\har_\delta = \bigoplus_k \har_{\delta}^k  $$  where $\har_{\delta}^k \subset \har_d^k \textrm{ are the harmonic forms in degree } k $. Furthermore, in a particular bidegree, they show $$\har_\del^{p,q} \cap \har_{\mub}^{p,q} = \har_d^{p,q} = \har_\delta^{p,q}.$$

Following Michelsohn \cite{Mi} we introduce a bidegree decomposition of the complex Clifford bundle $ \Cl(M) \stackrel{\textrm{def}}{=} Cl(M) \otimes \C$ $$\Cl(M) = \bigoplus_{r,s}\Cl^{r,s}(M)$$ and an $sl(2)$ representation on $\Cl(M)$ generated by operators $\Lt, \Lta, \Ht$. Furthermore there is an operator $\trs$ called the \textit{transpose map} inducing complex linear isomorphisms $$\trs: \Cl^{r,s}(M) \xrightarrow{\sim} \Cl^{r,-s}(M).$$

Generalizing Michelsohn's results in \cite{Mi}, we introduce two elliptic differential operators $\trD, \D$ and $\D^\trs = \trs \D \trs$ on sections of $\Cl(M)$. Additionally we define an operator $\boldsymbol{\hd} = \del - i\rhod$ on $\Omega_{\C}(M)$ where $\rhod$ is a zero order operator vanishing if and only if $\del \omega = \delb \omega = 0$. We prove local commutation relations between these operators and the generators of the respective $sl(2)$ representations. We show

\begin{propintro}[\Cref{ciso}] Let $M$ be a compact almost Hermitian manifold. Complex conjugation $c$ on $\Cl(M)$ induces complex anti-linear isomorphisms
$$ c: \har_{\D}^{r,s} \cap \har_{\D^\trs}^{r,s} \rightarrow \har_{\D}^{-r,-s}\cap \har_{\D^\trs}^{-r,-s} ,$$ 
$$ c: \har_{\trD}^{r,s} \longrightarrow \har_{\trD}^{-r,-s}. \hspace{.97in}$$
\end{propintro}

\begin{propintro}[\Cref{trsiso}] Let $M$ be a compact almost Hermitian manifold. The transpose map $\trs$ induces complex linear isomorphisms
$$ \trs: \har_{\D}^{r,s} \cap \har_{\D^\trs}^{r,s} \rightarrow \har_{\D}^{r,-s}\cap \har_{\D^\trs}^{r,-s}, $$
$$ \trs: \har_{\trD}^{r,s} \longrightarrow \har_{\trD}^{r,-s}. \hspace{.9in}$$
\end{propintro}

\begin{thmintro}[\Cref{AHsl2har}] Let $M$ be a compact almost Hermitian manifold. The Lie algebra generated by $\big ( \Ht, \Lt, \Lta \big)$ on $\Gamma \Cl(M)$ defines a finite dimensional $sl(2)$ representation on the space $\har_{\D} \cap \har_{\D^\trs}.$
\end{thmintro}

\begin{thmintro}[\Cref{AHcor}] Let $M$ be a compact almost Hermitian manifold. Through the isomorphism $\Gamma \Cl(M) \cong \Omega_\C(M)$ we have $$ \har_{\D} \cap \har_{\D^\trs}  \cong \har_{\hd} \cap \har_{\hdb}.$$
\end{thmintro}

\begin{thmintro}[\Cref{AKsl2har}] Let $M$ be a compact almost K{\"a}hler manifold. The Lie algebra generated by $\big ( \Ht, \Lt, \Lta \big)$ on $\Gamma \Cl(M)$ defines a finite dimensional $sl(2)$ representation on the space $\har_\trD$.
\end{thmintro}

\begin{thmintro}[\Cref{AKcor}] Let $M$ be a compact almost K{\"a}hler manifold. Through the isomorphism $\Gamma \Cl(M) \cong \Omega_\C(M)$ we have $$\har_\trD \cong \har_\delta  .$$
\end{thmintro}

\begin{corintro}[\Cref{Hodegehar}] On any compact almost Hermitian manifold $M$ there is a map $g$, called the \textit{Hodge automorphism},  inducing an isomorphism $$ \har_{\D}^{q-p,n-p-q} \cap \har_{\D^\trs}^{q-p,n-p-q} \stackrel{g}{\cong} \har_{\hd}^{p,q}\cap \har_{\hdb}^{p,q} .$$ Moreover, if $M$ is almost K{\"a}hler we have that $$\har_{\trD}^{q-p, n - p -q} \stackrel{g}{\cong} \har_{\delta}^{p ,q}. $$
\end{corintro}

\subsection*{Layout} 

In Section 2 we gather the necessary algebraic material on Clifford algebras, following and expanding on \cite{Mi}. Section 3 elaborates on the multiplicative interplay between Dirac operators defined using Hermitian connections and Michelsohn's $sl(2,\mathbb{C})$ operators. In Sections 4 and 5 we discuss canonical Hermitian connections \`a la Gauduchon and establish some results involving Dirac operators defined by these connections. Sections 7 and 8 prove the theorems mentioned above and transfer the results to operators on forms. The appendix provides an exposition of some fundamental results in almost Hermitian geometry, following mostly \cite{gau}.

\subsection*{Acknowledgements} This work constituted the author's PhD thesis at the City University of New York under the supervision of Luis Fernandez, whom he thanks for his patient guidance.

\section{The Clifford Algebra}

Let $V$ be a finite dimensional inner product space. We define the \textit{Clifford Algebra
of $V$}, denoted by $Cl(V)$, to be the vector space $\Lambda(V)$, with multiplication given
by $$v\cdot
\varphi = v \wedge \varphi - v \inter \varphi = E_v(\varphi) - I_v(\varphi)$$
for any $v \in V$ and $\varphi \in Cl(V)$, where $E_v^* = I_v$ is the adjoint of left exterior multiplication by $v$.

Let $V^\vee$ be the dual space to $V$. There is an \textit{algebra isomorphism} $$\Lambda(V) \cong \Lambda(V^\vee)$$ while, in general, the most one can hope for is a \textit{vector space isomorphism}
\begin{center} $Cl(V) \cong \Lambda(V^\vee).$ \end{center}

\subsection{Two Useful Involutions}

The \textit{antipodal map} $ \alpha : V \rightarrow V$ given by  $\alpha(v) = -v$ extends to an algebra automorphism of $Cl(V)$ by  $\alpha(v_1 \cdots v_p)= \alpha(v_1) \cdots \alpha(v_p)$. This is made precise by noting that $\alpha$ is an isometry of $V$, which are one-to-one with algebra automorphisms of $Cl(V)$ preserving $V$.

The \textit{tranpose map} is the algebra anti-automorphism $\trs$ of $Cl(V)$ reversing the order of multiplication, so that $\trs(v_1 \cdots v_p) = v_p \cdots v_1$ and $\trs(v) = v$ for $v \in V$. It is worth noting that the Hodge star operator $\ast$ on forms \big(sections of $\Lambda(V)$\big) takes the shape $\ast(\varphi) = \trs \alpha (\varphi)\cdot \textrm{vol}$ 
where $\textrm{vol} = v_1 \cdots v_{\dim(V)}.$

\subsection{Almost Complex Structures}

An \textit{almost complex structure} $J$ on a real vector space $V$ is an operator on $V$ so that $J^2 = - \textrm{Id}$. Such an operator bestows $V$ with a scalar multiplication making it a complex vector space, and conversely every complex vector space has a $J$ given by scalar multiplication by $i=\sqrt{-1}$. There are two useful ways to extend an orthogonal $J$ as an endomorphism of
$Cl(V)$. One way is as an \textit{algebra automorphism}  $$\Ja(v_1 \cdots v_p) = J(v_1) \cdots J(v_p) \hspace{.2in}$$
the other is as a \textit{derivation of the algebra}  $$\Jd(v_1 \cdots v_p) = \sum_j v_1 \cdots Jv_j \cdots v_p .$$These two extensions are related as follows: The family of orthogonal transformations of $V$ given by $J_t = \cos(t) \textrm{Id} + \sin(t)J$ induce a family of algebra automorphisms $\Ja(t)$ of $Cl(V)$. Differentiating $\Ja(t)$ at the identity one obtains $\Jd$. 
Let $e_1, \dots, e_n, Je_1,\dots, Je_n$ be an orthonormal basis. The element $\omega \in \Lambda^2(V^\vee)$ defined by $$\omega(v,w) = \langle Jv, w \rangle$$ can be viewed as an element of $Cl(V)$ by writing $$ \omega = \sum_j e_j \cdot Je_j $$ sometimes referred to as the \textit{fundamental 2-form}.

\subsection{The Complex Clifford Algebra}

We recall the usual orthogonal direct sum decomposition of the complexification of $V$ $$V \otimes \C = V^{1,0} \oplus V^{0,1}$$ where $V^{1,0}$ is the $i$-eigenspace of $J$ complex linearly extended to $V \otimes \C$ and $V^{0,1} = \ba{V^{1,0}}$.  We define $$\Cl(V) = Cl(V) \otimes \C.$$  When $V$ has an orthogonal almost complex structure $J$, a $J$-adapted orthonormal basis $e_1, \dots, e_n, Je_1, \dots, Je_n$ of $V$ gives rise to an orthogonal basis of $\Cl(V)$ $$\e_1, \dots, \e_n, \ba{\e}_1, \dots, \ba{\e}_n$$ where $\e_j =
\tfrac{1}{2}(e_j - iJe_j)$ and $\ba{\e}_j = \tfrac{1}{2}(e_j + i
Je_j)$ for $j = 1, \dots, n.$

All of these complex vectors anti-commute in $\Cl(V)$ except for pairs of the form $\e_k,
\ba{\e}_k$ where one has the remaining relation $$\e_k\ba{\e}_k + \ba{\e}_k\e_k = -1$$

\subsection{Michelsohn's Algebraic Results}

Michelsohn defined the operators $\Lt$ and $\Lta$ for any  $\varphi \in \Cl(V)$
by $$ \Lt (\varphi) = - \sum_k \e_k\cdot \varphi \cdot \ba{\e}_k
\hspace{.5in} \textrm{and}  \hspace{.5in} \Lta(\varphi) = - \sum_k \ba{\e}_k \cdot \varphi \cdot
\e_k .$$
and defined their commutator $$\Ht = [\Lt, \Lta].$$  

\begin{theorem}\cite{Mi} The operators $\Lt, \Lta, $ and $\Ht$ satisfy the relations $$[
\Lt, \Lta ] = \Ht,\ \   [\Ht, \Lt ] = 2 \Lt, \  \textrm{ and } [ \Ht, \Lta] =
-2\Lta $$ \end{theorem} In particular the subalgebra of
$\textrm{End}(\Cl(V))$ generated by the operators $\Lt,
\Lta, \Ht$ is an
$sl(2, \C)$ representation on $\Cl(V)$.

Letting $\om = \tfrac{1}{2i} \omega $ and $\J = -i \Jd$ Michelsohn showed that for any $ \varphi \in \Cl(V)$ \begin{center} $\Ht(\varphi) = \om \cdot \varphi + \varphi \cdot \om $. \end{center}  \begin{center} $ \J (\varphi) = \om \cdot \varphi - \varphi \cdot \om$ \end{center} 

in particular $\J$ and $\Ht$ commute. Furthermore since
$$ \displaystyle \J \Lt(\varphi) = -\sum_j \J(\e_j)\cdot \varphi \cdot \ba{\e}_j + \Lt \J(\varphi) - \sum_j \e_j \cdot \varphi \cdot \J(\ba{\e}_j) = \Lt \J(\varphi)$$ we have that $\J$ commutes with all of the generators of the $sl(2)$ representation on $\Cl(V)$.

We define $$\Cl^{r,s}(V) = \left \{ \varphi \in \Cl(V) \ |\  \J(\varphi) = r
  \varphi \textrm{ and } \Ht(\varphi) = s \varphi \right \}  $$ and there is a
bigrading $$ \displaystyle \Cl(V) = \bigoplus_{r,s = -n}^n \Cl^{r,s}(V).$$

We state some facts regarding $\Cl^{r,s}(V)$, the last of which we state as a Theorem (All borne from Michelsohn's article \cite{Mi})

\begin{itemize}
    \item $\Cl^{r,s}(V) = \{0\}$ if $r, s > n$ or if $ r, s < -n$ 
    \item $\Cl^{r,s}(V) = \{0\}$ for $r+s$ not congruent to $n = \tfrac{1}{2}\dim(V)$ mod 2. 
    \item $J$ on $V$ induces $-J^\vee$ on $V^\vee$. Through $Cl(V) \cong \Lambda(V^\vee)$ we have $\Jd = -(J^\vee)_{\textrm{der}}$

\end{itemize}

\label{curlyH=i(Lambda-L)}
\begin{theorem}\cite{Mi}  Through the isomorphism $\Cl(V) \cong \Lambda_\C(V^\vee)$ induced by complexification we have $$ \displaystyle \bigoplus_{s=-n}^n\Cl^{r,s}(V) \cong \bigoplus_{r= q-p }
    \Lambda_\C^{p,q}(V^\vee).$$  \end{theorem}

Recall the operators $L , \Lambda, H$, defined on $ \Lambda_{\C}(V^\vee)$ by $$L(\varphi) = \omega \wedge \varphi ,   \hspace{.2in} \Lambda = L^*  ,\hspace{.2in} \textrm{and} \hspace{.2in}
H(\varphi) = \sum_{p=0}^{2n}(n-p) \Pi_p(\varphi) .$$ 
Here $H= [\Lambda, L]$ and $\Pi_p: \Lambda_\C(V^\vee) \rightarrow \Lambda_\C^p(V^\vee) $ is the usual
projection.

\begin{theorem}\cite{Mi}
  Through the isomorphism $  \Cl(V) \cong \Lambda_{\C}(V^\vee)$ one has $$\Ht
  = i(\Lambda - L) \hspace{.2in} \Lt + \Lta = \alpha H
  \hspace{.2in} \Lt - \Lta =-i\alpha \big ( \Lambda + L\big )$$  where $\alpha$ is the antipodal involution. \end{theorem}

\begin{theorem}\label{g:hodgeaut}\cite{Mi} Let $g = \ex(\tfrac{-\pi i }{\phantom{-}4}H)
\ex(\tfrac{\pi}{4}(\Lambda-L)) \in SL(2, \C)$. Then we have the
following identities on $\Cl(V) \cong \Lambda_\C(V^\vee)$
$$ \Ht = g H g^{-1} \hspace{.5in} \alpha \Lt
= g \Lambda g^{-1}
\hspace{.3in}\textrm{and } \alpha \Lta = g L g^{-1} $$ \end{theorem}

The operator $g$ she called the \textit{Hodge automorphism}. By the previous Theorem we also write 
\begin{center} $g = \ex\big(-\tfrac{\pi i}{4}\alpha(\Lt + \Lta)\big) \ex\big(-\tfrac{\pi i}{4} \Ht \big).$
\end{center}

\begin{theorem}The Hodge automorphism induces an isomorphism \begin{center} $g^{-1}:\Cl^{q-p,n-p-q}(V) \xrightarrow{\sim} \Lambda^{p, q}(V^\vee) .$ \end{center} \end{theorem}
\begin{proof} Suppose $\varphi \in \Cl^{r,s}(V) \subset \Cl(V)  \cong \Lambda_{\C}(V^\vee)$ with $n \equiv r+s$ mod 2. $n -(r+s) = 2p$ and $n - (s-r) =2q$ for some $p,q$. I.e. $r= q-p$ and $s= n-p-q$. We have
\begin{equation*}
    \begin{split}s g^{-1} \varphi &= g^{-1} \Ht \varphi = g^{-1} g H g^{-1} \varphi = H g^{-1} \varphi \\ 
    r g^{-1} \varphi &= g^{-1} \J \varphi = \J g^{-1} \varphi = -\J^\vee g^{-1} \varphi
    \end{split}
\end{equation*} The result follows by observing
\begin{equation*} \Lambda^{p,q}(V^\vee) = \left \{ \psi \in \Lambda_\C(V^\vee) \ | \ \J^\vee(\psi) = -r\psi \textrm{ and } H(\psi) = s \psi \right \}. \qedhere \end{equation*}  \end{proof}

\begin{definition} Let $M$ be a Riemannian manifold. The vector bundle $Cl(M)$ is the associated bundle to the tangent bundle $TM$ with fibers
$Cl(T_x(M))$ for every $x \in M$. We define the \textit{complex Clifford bundle} $\Cl(M)$ by $$\Cl(M) = Cl(M) \otimes \C.$$  If $\nabla$ is an affine metric connection, $\nabla$ extends as a derivation over $\Gamma(Cl(M))$ and $\Gamma \big(\Lambda(M))$.  \end{definition}

\subsection{Some Vector Valued 2-forms}
Take $M$ to be an \textit{almost Hermitian manifold}, i.e. $M$ is Riemannian with an orthogonal $J$. We will be interested in affine metric connections $\nabla$ so that $\nabla(J) = 0$. Such connections are somewhat confusingly said to be \textit{Hermitian}. 

The \textit{torsion} of an affine connection $T = T_\nabla$ is given by $$  T(X,Y) = \nabla_X(Y) - \nabla_Y(X) - [X,Y]$$  and the \textit{Nijenhuis tensor}, $N$, is defined as
$$ N(X,Y) =\tfrac{1}{4}( [JX,JY] - J[JX,Y] - J[X,JY] - [X,Y]).$$

\subsection{The Riemannian Dirac Operator}

\begin{definition}Let $v_1, \dots, v_{2n}$ be a local orthonormal frame of $T(M)$ and let $\levic$ be the Levi-Civita connection on $\Gamma(TM)$. We define the \textit{Riemannian Dirac operator} on any $ \varphi \in \Gamma(Cl(M))$ by
$$ \rD(\varphi) = \sum_j v_j \cdot \levic_{v_j}\varphi .$$

\end{definition}

Let $d$ be the exterior derivative on $\Omega(M)$.

\begin{theorem} Through the isomorphism $\Gamma(Cl(M)) \cong \Omega(M)$ we have $$ \rD = d + d^* .$$ \end{theorem}

Let $\pi_{p,q}: \Omega_\C(M) \rightarrow \Omega_\C^{p,q}(M)$ be the natural projection onto bidegree. Complex linearly extending the exterior derivative $d$ to $\Omega_\C(M)$ we have a decomposition  $$d = \del + \mub + \delb + \mu$$ where for $\displaystyle \varphi \in \Omega^{p,q}(M)$ $$\displaystyle \del \varphi = \pi_{p+1,q} \circ d (\varphi) ,\hspace{.2in} \delb \varphi = \pi_{p,q+1}\circ d (\varphi),\hspace{.2in}  \mu \varphi= \pi_{p+2,q-1} \circ d (\varphi),\hspace{.2in}  \mub \varphi = \pi_{p-1,q+2} \circ d (\varphi) .$$ Let $\delta = \del + \mub$ and $\deltab = \delb + \mu$ so that on the complexification $\Gamma(\Cl(M)) \cong \Omega_{\C}(M)$ we have

$$ \rD = \delta + \deltab + \delta^* + \deltab^* .$$

\begin{proposition}\label{d^2}\cite{CiWi2} For an almost complex manifold $M$, the following relations among the operators $\del, \delb, \mu, \mub$ along with their adjoint identities obtain:
\begin{equation*}
    \begin{split}
        \mu^2 &= 0 \\
        \mu \del + \del \mu &=0 \\
        \mu \delb + \delb \mu + \del^2 &= 0 \\
        \mu \mub + \mub \mu + \del \delb + \delb \del &= 0 \\
        \mub \del + \del \mub + \delb^2 &= 0 \\
        \mub \delb + \delb \mub &= 0 \\
        \mub^2 &=0 .
    \end{split}
\end{equation*} \end{proposition}

\section{Almost Hermitian Dirac Identities}

\begin{definition} Let $\nabla$ be a Hermitian connection. We define a \textit{Hermitian Dirac operator} on $\Gamma(Cl(M))$ by

$$ D(\varphi) = \sum_j v_j \cdot \nabla_{v_j}\varphi . $$ \end{definition}

For an operator $T$ on $\Gamma(Cl(M)) \cong \Omega(M)$, we denote $$T_c
\stackrel{\textrm{def}}{=} \Ja^{-1}T \Ja \hspace{.1in} \textrm{ and } T^\trs
\stackrel{\textrm{def}}{=} \trs T \trs .$$

On $\Cl(M)$ 
\begin{equation*}
 \Lt_c = \Lta,\hspace{.15in} \Ht_c = \Ht,\hspace{.15in} \J_c = \J,\hspace{.15in} \J^\trs = \J,\hspace{.15in} \Ht^\trs = -\Ht, \hspace{.15in} \Lt^\trs = \Lta.
\end{equation*}

\begin{theorem}\label{[DH]}
Let $D$ be a Hermitian Dirac operator. Then on any almost Hermitian manifold $M$ we have the following \textbf{Almost Hermitian Dirac identities} on $\Gamma(\Cl(M))$
\begin{equation*}
\begin{split} [D, \Ht] =\phantom{-}[D, \J]\ \ = -iD_c \hspace{.2in} & \hspace{.2in}
  [D_c, \Ht] =\phantom{-}[D_c, \J] = \phantom{-}iD \\
[D^\trs, \Ht] =-[D^\trs, \J] = \phantom{-} iD_c^\trs \hspace{.2in}  &\hspace{.2in}
  [D_c^\trs, \Ht] = -[D_c^\trs, \J] = -iD^\trs .
\end{split}
\end{equation*} 
\end{theorem}
\begin{proof} For any $\varphi \in \Gamma( \Cl(M) )$, we have
$$ [D, \Ht]\varphi = D(\om \cdot \varphi) + D(\varphi \cdot \om) - \om \cdot D(\varphi)
- D(\varphi) \cdot \om .
$$ As $\nabla$ is metric and $J$ parallel, we have $\nabla(\om) = 0$. Using that $\nabla$ is a derivation, and that $$\om \cdot v_j - v_j \cdot \om = -i Jv_j$$ we obtain

\begin{equation*}
\begin{split} & D(\om \cdot \varphi) = \sum_j e_j \cdot \om \cdot
  \nabla_{e_j}(\varphi) + \sum_j Je_j \cdot \om \cdot
  \nabla_{Je_j}(\varphi), \hspace{.1in} \textrm{ and similarly} \\
& D(\varphi \cdot \om) = \sum_j e_j \cdot
  \nabla_{e_j}(\varphi) \cdot \om + \sum_j Je_j \cdot
  \nabla_{Je_j}(\varphi) \cdot \om = D(\varphi) \cdot \om 
\end{split}
\end{equation*} which gives the first identity. The remaining identities are obtained by conjugating with the operators $\trs$ and $ \Ja$. \end{proof}

\begin{remark} Recall $\Jd
\stackrel{\textrm{def}}{=} i \J$ is the usual extension of $J$ to a
derivation on the bundles $\Cl(M)\cong \Lambda_{\C}(M)$. We observe that $[d, \Jd] = d_c$ if and only if the manifold $M$ is
complex. 

One can check this by noting that $$[\mu, \Jd] = -3i \mu
\hspace{.1in}\textrm{ while }\hspace{.1in} \Ja^{-1} \mu \Ja = \phantom{-}i \mu .$$ Similarly $$[\mub, \Jd] = \phantom{-}3i \mub \hspace{.1in}\textrm{ while }\hspace{.1in} \Ja^{-1} \mub \Ja = -i \mub .$$ The expressions $$[\del, \Jd] = -i\del =
\Ja^{-1}\del\Ja \hspace{.1in}\textrm{ and }\hspace{.1in}[\delb, \Jd] = i\delb =
\Ja^{-1}\delb\Ja$$ agree in general.

Contrarily, we've observed above that $[D, \Jd]= D_c$ for any almost-Hermitian manifold. \end{remark}

\section{The Canonical Hermitian Connections}

Let $M$ be an almost Hermitian manifold of dimension $2n$. We denote by
$\Omega^2(TM)$ the space of vector valued $2$ forms on $M$. That is  $$ \Omega^2(TM) \stackrel{\textrm{def}}{=} \Gamma \big( \textrm{Hom}(\Lambda^2(M), TM) \big).$$ This space has many important elements, including the \textit{torsion} of a connection and the Nijenhuis tensor, $N$. It is a well known theorem of Newlander and Nirenberg \cite{Ne_Ni} that an almost complex manifold is complex if and only if $N = 0$. We observe the following less well known fact, cf. \cite{CiWi2},

\begin{lemma} The Nijenhuis tensor, $N$, complex bilinearly extended, is the
dual of $\mu + \mub$ on 1-forms. That is $\mu + \mub = N^\vee$.\end{lemma}

The space $\Omega^2(TM)$ is often identified through the metric isomorphism $TM \cong TM^\vee$ with the space of 3 tensors on $M$ which are skew-symmetric in the last two entries. More explicitly if $ \phi \in \Omega^2(TM) $ then for any $Y,Z \in TM$ we have $\phi(Y,Z) \in TM$. Through the metric isomorphism $TM \cong TM^\vee$ we have $\phi(Y,Z) \cong \langle \cdot, \phi(Y,Z) \rangle $. We identify $$\phi(X,Y,Z) = \langle X, \phi(Y,Z) \rangle .$$ In particular, this identification through the metric allows us to view other important objects as elements of $\Omega^2(TM)$. One special element of interest in almost Hermitian geometry is the 3-form $d\omega(X,Y,Z)$ which vanishes if and only if $M$ is \textit{almost K{\"a}hler}. Another is $(\levic\omega)(X,Y,Z) = \levic_X\omega(Y,Z)$ which can be rewritten in the following form 

$$(\levic_X \omega)(Y,Z) = \left \langle (\levic_X J)Y,Z \right \rangle.$$ 
 There are two distinguished projections of $\Omega^2(TM)$ onto differential forms of degree 1 and 3. There is the map $\bp : \Omega^2(TM) \rightarrow \Omega^3(M) $ defined by $$\bp(\phi)(X,Y,Z)= \tfrac{1}{3}\big(\phi(X,Y,Z) + \phi(Y,Z,X) + \phi(Z,X,Y) \big) \hspace{.7in}$$ 
and the map $r : \Omega^2(TM) \rightarrow \Omega^1(M)$ defined by $$r(\phi)(X) = \sum_j \phi(v_j,v_j,X)$$

In fact, \cite{gau} showed that the space $\Omega^2(TM)$ canonically decomposes as $$\Omega^2(TM) \cong \Omega^1(M) \oplus \Omega^3(M) \oplus \big(\ker(r) \cap \ker(\bp)\big).$$ (See \Cref{decomp_forms} for a proof.) We write the component of $\phi \in \Omega^2(TM)$ both the kernel of $r$ and $\bp$ as $\phi_0$.

The space of 3-forms $\Omega^3(M)$ is decomposed into
 $$ \ \  E^{+}  = \textrm{Re} \big( \Omega^{2,1}(M) \oplus \Omega^{1,2}(M) \big ),  $$
$$ E^{-} = \textrm{Re} \big( \Omega^{3,0}(M) \oplus \Omega^{0,3}(M) \big ) $$   so that $\Omega^3(M) = E^{-} \oplus E^{+}.$  For a 3-form $\varphi \in \Omega^3(M)$ we set $\varphi^+$ and $\varphi^-$ to be its components in $E^+$ and $E^-$ respectively.

\label{r(N)=0}

\begin{lemma} Let $N$ be the Nijenhuis tensor. Then $r(N) = 0$ and so $N = \bp(N) + N_0 .$ \end{lemma}

\begin{proof} Let $v_1, \dots, v_{2n} = e_1, \dots, e_n, Je_1, \dots ,Je_n$ be a $J$ adapted local orthonormal frame. Then for any $X \in TM$ we have

$$ r(N)(X) = \sum_{j=1}^{2n} N(v_j,v_j, X) = \sum_{j=1}^{2n} \langle v_j, N(v_j,X) \rangle $$

But $$\sum_{j=1}^{2n} \langle v_j, N(v_j,X) = \sum_{j=1}^{n} \langle e_j, N(e_j, X) \rangle + \sum_{j=1}^{n} \langle Je_j, N(Je_j,X) \rangle $$

Using that $N(Jv,w) = -JN(v,w)$ and that $J$ is orthogonal we obtain 

\begin{equation*} \sum_{j=1}^{n} \langle e_j, N(e_j, X) \rangle + \sum_{j=1}^{n} \langle Je_j, N(Je_j,X) \rangle = \sum_{j=1}^{n} \langle e_j, N(e_j, X) \rangle - \sum_{j=1}^{n} \langle e_j, N(e_j,X) \rangle = 0. \qedhere \end{equation*} \end{proof}

Perhaps the most important element of $\Omega^2(TM)$, for our purposes, is the potential $A^\nabla$ of an affine metric connection $\nabla$, defined by
 $$A^\nabla(X,Y,Z) = \big \langle \nabla_X(Y) - \levic_X(Y) , Z \big \rangle .$$ We shall drop the superscript $\nabla$ of the potential and torsion when the referent connection is clear.

 $A^\nabla$ is related to the torsion $T^\nabla$ by the following identity (See \Cref{A+T})
$$ A^\nabla + T^\nabla = 3\bp(A^\nabla) = \tfrac{3}{2}\bp(T^\nabla)$$ 

Gauduchon \cite{gau} defined an affine line of `canonical' Hermitian connections on an almost Hermitian manifold. This is the set of Hermitian connections, $\nabla^t$, uniquely defined by their torsion $T^t=T^{\nabla^t}$ satisfying 
\begin{equation} T^t = N + \tfrac{3t-1}{4}d_c\omega^+ - \tfrac{t+1}{4}\M(d_c\omega^+). \end{equation}

 Where $\M$ is the operator on $\Omega^2(TM)$ defined by $$\M(\phi)(X,Y,Z) = \phi(X,JY,JZ).$$
 For example, a natural choice of Hermitian connection is the \textit{Chern connection} $\nabla^{\textrm{Ch}}$ characterized on a \textit{Hermitian manifold} by the projection onto the $(0,1)$ component agreeing with the Dolbeault Operator. It is obtained by setting $t=1$.
 
 In the case that $M$ is a K{\"a}hler manifold, we have the well known identity $\nabla^{\textrm{ch}} = \levic $, which is true for all $\nabla^t$. Moreover all of the canonical Hermitian connections agree in the case that $d \omega = 0$, despite being distinct from the Levi-Civita connection.

 Another natural choice is the \textit{Bismut connection} $\nabla^{\textrm{Bm}}= \nabla^{-1}$ characterized on a Hermitian manifold by its torsion -- viewed as a $3$-tensor -- being totally skew symmetric. For all $t\in \mathbb{R}$ one has
$$ \nabla^t = \tfrac{1+t}{2} \nabla^1 + \tfrac{1-t}{2} \nabla^{-1}. $$

We denote the potential of a canonical Hermitian connection by $A^t$. By the above relation between the torsion and the potential we observe $$A^t = -T^t + \tfrac{3}{2} \bp T^t .$$ We then have

\begin{proposition}\label{A^t=N+dometcgau}\cite{gau} Let $M$ be an almost Hermitian manifold and let $\nabla^t$ be a canonical Hermitian connection in the sense of Gauduchon. Then $$A^t = -N + \tfrac{3}{2}\bp N + \tfrac{t-1}{4}d_c\omega^+ + \tfrac{t+1}{4}\M(d_c\omega^+).$$\end{proposition}

\section{The Canonical Hermitian Dirac Operators}

\begin{definition} Given a local orthonormal frame $v_1, \dots, v_{2n}$ of $TM$ and a $t \in \R$ we define the \textit{canonical Hermitian Dirac operator} $D_t$ by $$ D_t = \sum_j v_j \cdot \nabla^t_{v_j}.$$

 For any $v \in TM$ and any $\varphi \in \Gamma( Cl(M) )$ we define $$\A^t_v(\varphi) \stackrel{\textrm{def}}{=} \nabla^t_v(\varphi) - \levic_v(\varphi).$$ Evidently $\A^t_v$ is a derivation, and for vector fields $X,Y,Z$ we have that, by definition $$ A^t(X,Y,Z) = \left \langle \A^t_X(Y),Z \right \rangle .$$

 For the time being, we suppress the parameter $t$, taking $A = A^t$ and $\A = \A^t$. 

  We define the operator $\DA$ on $\Gamma(Cl(M))$ by $$ \DA = \sum_j v_j \cdot \A_{v_j}$$ so that $D_t = \rD + \DA $. The 1-form $r(A)$ is given, for any $v \in TM$, by $$r(A)(v) = \sum_jA(v_j,v_j,v) = \sum_j \langle \A_{v_j}(v_j), v \rangle$$
so that $r(A)^\sharp = \sum_j \A_{v_j}(v_j)$. Lastly, for any vector $v \in TM$ we define $L_v$ on $Cl(M)$ by $$ L_v(\varphi) = v \cdot \varphi$$
and we will write $L_{r(A)}$ to denote $L_{r(A)^\sharp}$. \end{definition}

\subsection{A Result Concerning Self-Adjointness of the Canonical Hermitian Dirac Operators}

 Since both $\nabla^t$ and $\levic$ are metric connections we have, for any $v \in TM$, that the adjoint $\A_v^* = - \A_v$. Furthermore for a unital vector $v$ we have $L_v$ is an isometry.

\begin{proposition}\label{DA^*} On any almost Hermitian manifold we have $\DA^* = \DA + L_{r(A)}.$ \end{proposition}

\begin{proof} For any $\varphi, \psi \in \Gamma(Cl(M))$ we have 
\begin{equation*}
    \begin{split} \langle \DA(\varphi), \psi \rangle = \left \langle \sum_j v_j \cdot \A_{v_j}(\varphi), \psi \right \rangle &= -\sum_j \left \langle \A_{v_j}\varphi, v_j\cdot \psi \right \rangle \\
        &= \phantom{-} \sum_j \left \langle \varphi, \A_{v_j}(v_j \cdot \psi) \right \rangle \\
        &= \phantom{-} \sum_j \left \langle \varphi, \A_{v_j}(v_j) \cdot \psi \right \rangle + \left \langle \varphi, v_j \cdot \A_{v_j}\psi \right \rangle \\
        &= \phantom{-} \left \langle \varphi, (\DA + L_{r(A)})(\psi) \right \rangle. \qedhere
    \end{split}
\end{equation*}
\end{proof}

 We observe that $r(A) = r\big(-N + \tfrac{3}{2}\bp N + \tfrac{t-1}{4}d_c\omega^+ + \tfrac{t+1}{4}\M(d_c\omega^+) \big)$. Since $r(N) = 0$ by \Cref{r(N)=0}, and $r$ vanishes on any 3-form we obtain
\begin{equation}\label{r(A)=lee} r(A) = \tfrac{t+1}{4}r\big(\M(d_c\omega^+)\big).  \end{equation}

\begin{definition} We define the \textit{Lee form} to be the $1$-form $\theta = \Lambda(d\omega)$. \end{definition}

 Observe that, for bidegree reasons, $\Lambda(\mu \omega) = \Lambda(\mub \omega) = 0$ on $\Omega_{\C}(M)$, so that $\theta = \Lambda(d\omega^+)$. Furthermore, it is well known \cite{gau} that $\theta = J d^*\omega = \Lambda(d \omega) $ .

\begin{lemma} $r\big(\M(d_c\omega^+)\big) = 2 \theta .$ \end{lemma}

\begin{proof} Let $v_1, \dots, v_{2n} = e_1 ,\dots, e_n, Je_1, \dots Je_n$ be a $J$-adapted local orthonormal frame. For any $X \in TM$ we have 
\begin{equation*}
    \begin{split} r\big (\M(d_c \omega^+)\big)(X) &= \phantom{-2}\sum_{j=1}^{2n} \M(d_c\omega^+)(v_j,v_j,X) = \phantom{.}\sum_{j=1}^{2n} (d_c \omega)^+(v_j,Jv_j,JX) \\
    &= \phantom{-2}\sum_{j=1}^n (d_c \omega)^+(e_j,Je_j,JX) - \phantom{2}\sum_{j=1}^n (d_c \omega)^+(Je_j,e_j,JX) \\ 
    &= \phantom{-}2 \sum_{j=1}^n(d_c\omega)^+(e_j,Je_j,JX) = -2 \sum_{j=1}^n(d\omega)^+(Je_j,e_j,X) \\
    &= -2 \big(\omega \inter d\omega^+ (X) \big) \hspace{.1in} = \hspace{.1in} 2 \Lambda(d\omega^+)(X) \hspace{.1in} = \hspace{.1in} 2 \theta(X). \qedhere
    \end{split}
\end{equation*}

\end{proof}

Hence, by the identity \eqref{r(A)=lee} we have that
$$ r(A) = \tfrac{t+1}{2}\theta. $$ Together with \Cref{DA^*} we obtain $$ D_t^* - \rD^* = \DA^* = \DA + L_{r(A)} = D_t - \rD + \tfrac{t+1}{2} L_\theta.$$ As the Riemannian Dirac operator $\rD$ is self-adjoint we conclude the

\begin{proposition}\label{D^*} For any $t \in \R$ the adjoint of the canonical Hermitian Dirac operator $D_t$ is given by $$D_t^* = D_t + \tfrac{t+1}{2} L_\theta.$$  \end{proposition}

Following \cite{Mi2} we recall a condition on the metric determined by requiring $d^*\omega = 0$. Such metrics are said to be `balanced'. We note that the balanced condition obtains if and only if $ Jd^*\omega = \theta = 0$. Thus \Cref{D^*} implies the following

\begin{corollary} For all $t \neq -1$ we have $D_t$ is self-adjoint if and only if the metric is balanced. \end{corollary}

 For $t=-1$ the Dirac operator $D_{-1}$ is self-adjoint for any almost Hermitian manifold, sans any requirement on the metric. This directly corresponds to an observation made by Bismut \cite{Bis}, \cite{gau} regarding Hermitian Dirac operators on the spinor bundle induced by the connection $\nabla^{-1}$ being self-adjoint. Presently, the connection $\nabla^{-1}$ is widely referred to as the \textit{Bismut connection}. We define
$$\bD \stackrel{\textrm{def}}{=} D_{-1}.$$ 

 We also define the operator $\Dt$ on $\Gamma \big(\Cl(M)\big)$ by \begin{equation}\label{d_tcoord} \Dt \stackrel{\textrm{def}}{=} \tfrac{1}{2} \big( D_t + i (D_t)_c\big)  = 2\sum_j \e_j \cdot \nabla^t_{\ba{\e}_j}   \end{equation}
We observe (cf. \cite{Mi}) that the operator $\Dt$ is of Clifford bidegree $(1,1)$ as left Clifford multiplication by an element of $T^{1,0}(M)$ is of Clifford bidegree $(1,1)$ and any Hermitian connection is both $\Ht$ and $\J$ parallel.

 In her seminal paper \cite{Mi2} Michelsohn showed that for a \textit{Hermitian} manifold $M$ the operator $\Dch_1$ has adjoint $\ba{\Dch_1}$ if and only if the metric is balanced (Proposition 2.3 in that paper). We extend this result to the almost Hermitian setting and to all $\Dt$. By the above 
$$\Dt^* = \tfrac{1}{2} \left ( D_t^* -i (D_t)_c^* \right ) =\ba{\Dt} + \tfrac{t+1}{4} L_{\small \tfrac{\theta + i J\theta}{2}. \normalsize} $$ Since $\tfrac{1}{2}(\theta + i J \theta)$ is the projection of $\theta \in T_{\C}(M)$ to $T^{0,1}(M)$  we conclude

\begin{proposition}\label{curlyD_t self adjoint iff balanced} Let $M$ be a compact almost Hermitian manifold.  The operator $\D \stackrel{\textrm{def}}{=} \mathfrak{d}_{-1}$ is conjugate self-adjoint on $\Gamma(\Cl(M))$. Moreover, for all $t \neq -1$ $\Dt$ is conjugate self adjoint if and only if the metric is balanced.\end{proposition}

\subsection{Bochner Identities}

Using only that $\nabla^t$ is Hermitian, and the local expression \eqref{d_tcoord} of $\Dt$ via a local frame of $T_\C(M)$  we have 
\[ \tfrac{1}{4} \Dt^2 = \sum_{j,k} \e_j\cdot\e_k\cdot \big(\nabla^t_{\ba{\e}_j}\nabla^t_{\ba{\e}_k} - \nabla^t_{\nabla^t_{\ba{\e}_j}\ba{\e}_k}\big) \] which we rewrite as
\[ \tfrac{1}{4} \Dt^2 = \sum_{j<k} \e_j\cdot\e_k\cdot \big(\nabla^t_{\ba{\e}_j,\ba{\e}_k} - \nabla^t_{\ba{\e}_k,\ba{\e}_j}\big) \] where $$\nabla_{X,Y} \stackrel{\textrm{def}}{=} \nabla_X\nabla_Y - \nabla_{\nabla_{X}Y}$$ is the second covariant derivative with respect to an affine connection $\nabla$. The curvature tensor $R^\nabla$ defined by \[ R^\nabla(X,Y) = [\nabla_X, \nabla_Y] - \nabla_{[X,Y]} \] is related to the second covariant derivative and the torsion of $\nabla$ by 
\[ R^\nabla(X,Y) = \nabla_{X,Y} - \nabla_{Y,X} - \nabla_{T(X,Y)} \] so that 
\[ \tfrac{1}{4} \Dt^2 = \sum_{j<k} \e_j\cdot\e_k\cdot \big(R^t(\ba{\e}_j,\ba{\e}_k) + \nabla^t_{T^t(\ba{\e}_j,\ba{\e}_k)}\big). \]
Complex bilinearly extending the metric, we observe $$T^t(\ba{\e}_j,\ba{\e}_k) = \sum_i T^t(\e_i,\ba{\e}_j,\ba{\e}_k)\ba{\e}_i + T^t(\ba{\e}_i,\ba{\e}_j,\ba{\e}_k)\e_i$$ so that by the definition of a canonical Hermitian connection (1)

\[ T^t(\ba{\e}_j,\ba{\e}_k) = \sum_i  N(\ba{\e}_i,\ba{\e}_j,\ba{\e}_k)\e_i + \tfrac{3t-1}{4} d_c\omega^+(\e_i,\ba{\e}_j,\ba{\e}_k)\ba{\e}_i  - \tfrac{t+1}{4} \M(d_c\omega^+)(\e_i,\ba{\e}_j,\ba{\e}_k)\ba{\e}_i \] where we use that $N$ sends elements of $\Lambda^{0,2}TM$ to $T^{1,0}M$  and that $d_c\omega^+$ vanishes on sections of $\Lambda^{3,0}TM \oplus \Lambda^{0,3}TM $ (by definition of the $+$ component of a 3-form). Furthermore, by definition of $\M$ we observe \[ \M(d_c\omega^+)(\e_i,\ba{\e}_j,\ba{\e}_k) = d_c\omega^+(\e_i,J\ba{\e}_j,J\ba{\e}_k)=  - d_c\omega^+(\e_i,\ba{\e}_j,\ba{\e}_k) \] and thus \[ T^t(\ba{\e}_j,\ba{\e}_k) = \sum_i  N(\ba{\e}_i,\ba{\e}_j,\ba{\e}_k)\e_i + t d_c\omega^+(\e_i,\ba{\e}_j,\ba{\e}_k)\ba{\e}_i.  \] We conclude

\begin{proposition} Let $M$ be an almost Hermitian manifold. On $\Gamma(\Cl(M))$ we have
\[ \tfrac{1}{4} \Dt^2 = \sum_{j<k}  \e_j\cdot\e_k \cdot R^t(\ba{\e}_j,\ba{\e}_k) + \e_j\cdot \e_k \cdot \big(\sum_i N(\ba{\e}_i,\ba{\e}_j,\ba{\e}_k)\nabla^t_{\e_i} + td_c\omega^+(\e_i,\ba{\e}_j,\ba{\e}_k)\nabla^t_{\ba{\e}_i}\big). \] \end{proposition}

 Let $t=1$ and $N=0$, the curvature tensor $ R=R^1 $ of the Chern connection is of type (1,1) and hence $R(\ba{\e}_j,\ba{\e}_k) = 0$ for all $j,k$. In this case, the operator $\mathfrak{d}_1$ is a differential if and only if $d_c \omega = d_c\omega^+ = 0$, i.e. if and only if $M$ is K{\"a}hler. This was first shown by Michelsohn (\cite{Mi2} Proposition 2.1). Evidently if $M$ is K{\"a}hler then $\Dt = \mathfrak{d}_1 $ and so $\Dt$ is a differential. 

 A more involved, but completely analogous computation to the above reveals

\[ \Dt \Dct + \Dct \Dt = - \sum_j \nabla^t_{\ba{\e}_j,\e_j} + \sum_{j,k}\ba{\e}_k\cdot\e_j\big( R^t(\e_k,\ba{\e}_j) - \nabla^t_{T^t(\e_k,\ba{\e}_j)} \big). \] Following Michelsohn \cite{Mi} we define the operators \[\bm{\nabla^*\nabla}_t =  - \sum_j \nabla^t_{\ba{\e}_j,\e_j} \textrm{ and } \mathcal{R}_t = \sum_{j,k}\ba{\e}_k\cdot\e_j \cdot R^t(\e_k,\ba{\e}_j) \] so that

\[ \Dt \Dct + \Dct \Dt = \bm{\nabla^*\nabla}_t + \mathcal{R}_t - \tfrac{t-1}{2} \sum_{i,j,k} \ba{\e}_k\cdot\e_j \cdot\big(d_c \omega^+(\e_i,\e_k,\ba{\e}_j) \nabla^t_{\ba{\e}_i} + d_c \omega^+(\ba{\e}_i,\e_k,\ba{\e}_j) \nabla^t_{\e_i} \big). \] For $t = -1$ this produces an elegant description of the Laplacian of $\D$ over a compact almost Hermitian manifold, while for $t = 1$ the terms involving $d_c \omega ^+$ vanish, yielding Michelsohn's Bochner identity in \cite{Mi} in the almost Hermitian setting.

\subsection{The Canonical Operators on Forms}
Proceeding with our study of canonical Hermitian Dirac operators, we begin with the following

\begin{definition} Let $v_1, \dots, v_{2n}$ be a local orthonormal frame. We define the operator $\DEA$ on $\Omega(M)$ by
$$\DEA = \sum_j v_j \wedge \A_{v_j} .$$ \end{definition} Since for any $v \in TM$ we have $\A_v$ is a derivation on $\Omega(M)$ it follows that $\DEA $ is a graded derivation on $\Omega(M)$. 

Furthermore, we observe for any $Y \in \Gamma(TM) \cong \Omega^1(M)$ \small \[ \DEA(Y) = \sum_j v_j \wedge \A_{v_j}(Y) = \sum_{j,k} \langle \A_{v_j}(Y),v_k \rangle v_j \wedge v_k = \sum_{j,k} A(v_j,Y,v_k) v_j \wedge v_k   \hspace{.6in} \] \normalsize and so by \Cref{A^t=N+dometcgau} we see that \small
\begin{equation}\label{dA=sumA}  \DEA(Y) =  \sum_{j,k} \Big( \textrm{-}(N - \tfrac{3}{2}\bp N)(v_j,Y,v_k) + \tfrac{t-1}{4}d_c\omega^+(v_j,Y,v_k) + \tfrac{t+1}{4}\M(d_c\omega^+)(v_j,Y,v_k) \Big)v_j \wedge v_k . \end{equation} \normalsize


 With regard to the adjoint of $\DEA$ we have the following

\begin{lemma} Let $M$ be an almost Hermitian manifold. For any $\varphi \in \Omega(M)$ the adjoint of $\DEA$ is given by $$ \DEA^*(\varphi) = -r(A) \inter \varphi-\sum_j v_j \inter \A_{v_j}(\varphi) .$$ \end{lemma}

\begin{proof} Let $\varphi \in \Omega(M)$. As interior multiplication by a vector is the adjoint of exterior multiplication, and $\A_v$ is anti-self-adjoint we have that $$ \DEA^*(\varphi) = -\sum_j \A_{v_j}(v_j \inter \varphi) .$$
 Using that, for any $v \in TM$, $\A_v$ is a $C^\infty$-linear derivation, it is also a derivation over interior multiplication by vectors. That is
$$ \A_v(v \inter \varphi) =  v \inter \A_v(\varphi) + \A_v(v) \inter \varphi .$$
Hence $$\sum_j \A_{v_j}(v_j \inter \varphi) = \sum_j v_j \inter \A_{v_j}(\varphi) + \sum_j \A_{v_j}(v_j) \inter \varphi $$
and the result follows. \end{proof}

\begin{proposition}\label{DA=dA+dA*+Ir(A)} For any $ \varphi \in \Gamma(Cl(M)) \cong \Omega(M)$ we have $$D_t(\varphi) - \rD(\varphi) = \DA(\varphi) \cong \DEA(\varphi) + \DEA^*(\varphi) + r(A)\inter \varphi .$$ \end{proposition}

\begin{proof} For any $\varphi \in \Gamma(Cl(M)) \cong \Omega(M)$ we have

\begin{equation*} \DA(\varphi) = \sum_j v_j \cdot \A_{v_j}(\varphi) \cong \sum_j v_j \wedge \A_{v_j}(\varphi) - \sum_j v_j \inter \A_{v_j}(\varphi) = \DEA(\varphi) + \DEA^*(\varphi) + r(A) \inter \varphi . \qedhere \end{equation*} \end{proof}

 We now define a host of operators on forms. Most of these operators will likely be familiar to the reader, with one exception. We begin by defining the operator $\lambda^+: \Omega(M) \rightarrow \Omega(M)$ for any $ \varphi \in \Omega(M)$ by $$ \lambda^+(\varphi) = d\omega^+ \wedge \varphi$$
and the operator $\tau^+$ on $\Omega(M)$ by
$$ \tau^+ = [\Lambda, \lambda^+].$$

 Lastly we define the novel operator $\rho^+: \Omega(M) \rightarrow \Omega(M)$, for any $\varphi = \varphi_1 \wedge \cdots \wedge \varphi_k \in \Omega^k(M)$, by $$ \rho^+(\varphi) = \sum_{j=1}^k (-1)^j \varphi_1 \wedge \cdots \wedge
(\varphi_j \inter d\omega^+) \wedge \cdots \wedge \varphi_k .$$
Note that, as usual, by $ \varphi_j \inter d \omega $ we mean $\varphi_j^\sharp
\inter d \omega$ where $\sharp: T^\vee(M) \rightarrow T(M)$ is the metric isomorphism. Letting $ v_1, \dots,v_{2n}$ be a $J$-adapted local orthornormal frame of $TM$, we can also express
$$ \rho^+(\varphi) = -\sum_j(v_j\inter d\omega^+) \wedge( v_j \inter \varphi) \hspace{.2in} \textrm{ for any } \varphi \in \Omega(M). $$

 Complex linearly extending $\rho^+$ we restrict to components of
$d\omega^+ = \del \omega + \delb \omega $ and
define for $ \varphi =  \varphi_1 \wedge \cdots \wedge \varphi_k \in
\Omega_{\C}^{p,q}(M) \subset \Omega_{\C}^k(M)$
$$\rho_\gamma (\varphi) =  \sum_{j=1}^k (-1)^j \varphi_1 \wedge \cdots
\wedge \left  (\varphi_j \inter \gamma\omega \right ) \wedge \cdots \wedge \varphi_k
\textrm{ for } \gamma = \del, \delb .$$
In the complexification by $ \varphi_j \inter d \omega $ we mean $\ba{\varphi_j^\sharp}
\inter d \omega$. We also define the complex linearly extended operators $\lambda_\gamma$ and $\tau_\gamma$ for $\gamma = \del, \delb $ by  $$\lambda_{\gamma} (\varphi)\stackrel{\textrm{def}}{=} \gamma (\omega) \wedge \varphi  \hspace{.5in} \textrm{ and }\hspace{.5in} \tau_{\gamma} (\varphi)\stackrel{\textrm{def}}{=} [\Lambda, \lambda_{\gamma}] \varphi.$$
It is quick to check that $\rho_\del,\  \rho_{\delb} $ have bidegree
$(1,0)$ and $(0,1)$ respectively, and that
$\rho_{\ba{\del}} = \ba{\rho_\del}$. Moreover, for $\varphi \in \Omega^{p,q}(M)$ we have $$ \Ja^{-1}\rhod \Ja (\varphi) = \frac{i^{p-q}}{i^{p+1-q}} \rhod (\varphi) = -i \rhod(\varphi)  \hspace{.1in} \textrm{ and } \hspace{.4in} \Ja^{-1}\rhodb \Ja (\varphi) = i \rhodb (\varphi). \hspace{.7in} $$Note that as operators on $\Omega_{\C}(M)$ we have $$ \rho^+ = \rho_\del + \ba{\rho_\del} \hspace{.5in} \textrm{ and } \hspace{.5in}  \tau^+ = \tau_{\del} + \ba{\tau_{\del}} .$$ By the last two remarks we also have $$\rho^+_c = \Ja^{-1}\rho^+\Ja = i (\rhodb - \rhod).$$

 The following result allows us to express our canonical Hermitian Dirac operators in terms of the above operators on forms

\begin{theorem}\label{gaudmainprop} For any almost Hermitian manifold $M$ we have through the canonical vector bundle isomorphism $\Omega_{\C}(M) \cong \Gamma(\Cl(M))$ that $$D_t = \del + \delb + \del^* + \delb^* + \tfrac{t+1}{4}(\taud + \taudb + \taud^* + \taudb^* - E_\theta + I_\theta) + \tfrac{3t-1}{4}i(\rhod - \rhodb -\rhod^* + \rhodb^*).$$ \end{theorem}

 In order to prove the theorem, we will require a few lemmas, starting with

\begin{lemma}\label{N-bN=} For $Y \in \Gamma( T_\C M ) \cong \Omega_{\C}^1(M)$ we have $$\sum_{j,k}(N-\tfrac{3}{2}\bp N)(v_j,Y,v_k)v_j\wedge v_k = (\mu + \mub)(Y).$$ \end{lemma}

\begin{proof}  By \Cref{r(N)=0} we observe $N = \bp N + N_0$ and so
$$ N - \tfrac{3}{2}\bp N = \bp N + N_0 - \tfrac{3}{2} \bp N = N_0 - \tfrac{1}{2} \bp N .$$
Thus for $Y \in \Gamma(TM) \cong \Omega^1(M)$ we have\small $$ \sum_{j,k}(N - \tfrac{3}{2}\bp N)(v_j, Y, v_k) v_j \wedge v_k = \sum_{j,k}(N_0 - \tfrac{1}{2}\bp N)(v_j, Y, v_k) v_j \wedge v_k . \hspace{.4in}$$ \normalsize
By definition $\bp N_0 = 0$, so that
\small
$$ \sum_{j,k} N_0(v_j, Y, v_k) v_j\wedge v_k = \sum_{j,k} N_0(Y,v_j,v_k) v_j \wedge v_k - \sum_{j,k} N_0(v_j,Y,v_k)v_j\wedge v_k $$ \normalsize
and thus \small $$ \sum_{j,k} N_0(v_j,Y,v_k)v_j\wedge v_k = \tfrac{1}{2} \sum_{j,k} N_0(Y,v_j,v_k) v_j \wedge v_k . \hspace{1.4in}$$
\normalsize
Finally using that $\bp N$ is a 3-form we conclude
\small
$$ \sum_{j,k}(N - \tfrac{3}{2}\bp N)(v_j, Y, v_k) v_j \wedge v_k = \tfrac{1}{2}\sum_{j,k}(N_0 + \bp N)(Y, v_j, v_k) v_j \wedge v_k  = \tfrac{1}{2} \sum_{j,k}N(Y,v_j,v_k)v_j \wedge v_k .$$
\normalsize
Observing that $N^\vee = \mu + \mub$ on $\Omega_\C(M)$ gives the result.\end{proof}

\begin{lemma}\label{(3)} For $Y \in \Gamma( TM ) \cong \Omega^1(M)$ we have $$ \sum_{j,k} \M (d_c\omega^+)(v_j,Y,v_k)v_j \wedge v_k = \tfrac{1}{2} \sum_{j,k}\big( d_c\omega^+(v_j,Y,v_k)v_j\wedge v_k + d\omega^+(v_j,JY,v_k) v_j \wedge v_k \big). $$ \end{lemma}
\begin{proof}  We have for any $\varphi^+ \in E^{+}$
$$\varphi^+(X,Y,Z) = \varphi^+(JX,JY,Z) + \varphi^+(X,JY,JZ) + \varphi^+(JX,Y,JZ) $$ (see \Cref{Jderid+} for a proof). Thus
\small
\begin{equation*}
    \begin{split}
        \sum_{j,k}\M (d_c\omega^+)(v_j,Y,v_k)v_j\wedge v_k &= 
        -\sum_{j,k}d\omega^+(Jv_j,Y,v_k)v_j\wedge v_k  \\
        &= -\sum_{j,k}\big(d\omega^+(Jv_j,JY,Jv_k) - d\omega^+(v_j,Y,Jv_k) - d\omega^+(v_j,JY,v_k)\big) v_j\wedge v_k . \\
    \end{split}
\end{equation*}
\normalsize
Subtracting the expression \small $\displaystyle \sum_{j,k}d \omega^+(v_j,Y,Jv_k)v_j\wedge v_k$ \normalsize from both sides of the last equality we obtain
\small
\begin{equation*}
    \begin{split}
        \sum_{j,k}\M (d_c\omega^+)(v_j,Y,v_k)v_j\wedge v_k &= 
        \tfrac{1}{2}\sum_{j,k}\big(-d\omega^+(Jv_j,JY,Jv_k) + d\omega^+(v_j,JY,v_k) \big)v_j\wedge v_k \\
        &= \tfrac{1}{2} \sum_{j,k}\big( d_c\omega^+(v_j,Y,v_k)v_j\wedge v_k + d\omega^+(v_j,JY,v_k) v_j \wedge v_k \big). \qedhere
    \end{split}
\end{equation*} 
\normalsize \end{proof}

\begin{lemma}\label{(4)} For $Y \in \Gamma( TM ) \cong \Omega^1(M)$ we have $$ \tfrac{1}{2} \sum_{j,k}d\omega^+(v_j,JY,v_k)v_j\wedge v_k = \tau^+(Y) - \theta \wedge Y . $$ \end{lemma}

\begin{proof}  Observe that for $Y \in \Gamma( TM ) \cong \Omega^1(M)$ $$\tau^+(Y) = [\Lambda, \lambda^+](Y) = - \Lambda(Y \wedge d\omega^+) = \sum_{j= 1}^n e_j \inter Je_j \inter (Y \wedge d\omega^+).$$
Using that interior multiplication is a graded derivation we obtain
\small
\begin{equation*}
    \begin{split} \sum_{j= 1}^n e_j \inter Je_j \inter (Y \wedge d\omega^+) &= \sum_j e_j \inter \big( \langle Y, Je_j \rangle d\omega^+ - Y \wedge (Je_j \inter d\omega^+)\big) \\
    &= \sum_j \big(\langle Y, Je_j \rangle e_j \inter d\omega^+ - \langle Y, e_j \rangle Je_j \inter d\omega^+ + Y \wedge(e_j \inter Je_j \inter d\omega^+) \big) \\
    & = \sum_j \big(-\langle JY, e_j \rangle e_j \inter d\omega^+ - \langle JY, Je_j \rangle Je_j \inter d\omega^+\big) + Y \wedge \sum_j e_j \inter Je_j \inter d\omega^+ \\
    &= -JY \inter d\omega^+ + \Lambda(d\omega^+)\wedge Y . 
        \end{split}
\end{equation*} 
\normalsize
As $\Lambda(d\omega^+) = \Lambda(d\omega) = \theta$, the result follows. \end{proof}

\begin{lemma}\label{(5)} For $Y \in \Gamma( TM ) \cong \Omega^1(M)$ we have $$ \tfrac{1}{2} \sum_{j,k}d_c\omega^+(v_j,Y,v_k)v_j\wedge v_k = - \rho_c^+(Y). $$ \end{lemma}

\begin{proof}  For $Y \in \Gamma(TM) \cong \Omega^1(M)$ we have $$\rho^+(Y) = \tfrac{1}{2}\sum_{j,k} d\omega^+(v_j,Y,v_k)v_j\wedge v_k$$
and hence 
\begin{equation*}
    \begin{split}
        \rho_c^+(Y) &= \tfrac{1}{2} \Ja^{-1}\big ( \sum_{j,k}d\omega^+(v_j,JY,v_k)v_j\wedge v_k\big) \\
        &= \phantom{-}\tfrac{1}{2}\sum_{j,k}d\omega^+(v_j,JY,v_k)Jv_j\wedge Jv_k \\
        &= \phantom{-}\tfrac{1}{2}\sum_{j,k}d\omega^+(Jv_j,JY,Jv_k)v_j\wedge v_k \\
        &= -\tfrac{1}{2}\sum_{j,k}d_c\omega^+(v_j,Y,v_k)v_j\wedge v_k . \qedhere
    \end{split}
\end{equation*}

\end{proof}

 Combining lemmas \cref{(3)}, \cref{(4)} and \cref{(5)} we obtain

\begin{proposition}\label{Mdcom=} For $Y \in \Gamma( TM ) \cong \Omega^1(M)$ we have $$ \sum_{j,k} \M (d_c\omega^+)(v_j,Y,v_k)v_j \wedge v_k = \tau^+(Y) - \theta \wedge Y - \rho_c^+(Y) . $$ \end{proposition}

 We now prove \Cref{gaudmainprop}:


\begin{proof} (of \Cref{gaudmainprop}) Complex linearly extending $\DEA$ to $\Omega_{\C}(M)$, we recall that for any $Y \in \Omega_{\C}^1(M)$

\small
\[ \DEA(Y) =  \sum_{j,k} \Big( \textrm{-}(N - \tfrac{3}{2}\bp N)(v_j,Y,v_k) + \tfrac{t-1}{4}d_c\omega^+(v_j,Y,v_k) + \tfrac{t+1}{4}\M(d_c\omega^+)(v_j,Y,v_k) \Big)v_j \wedge v_k . \] \normalsize
By \Cref{Mdcom=}, \Cref{N-bN=}, \Cref{(5)}  we have 
$$\DEA(Y) = -\big(\mu + \mub\big)(Y) + \tfrac{t+1}{4}\big(\taud(Y) + \taudb(Y) - \theta \wedge Y\big) +\tfrac{3t-1}{4}i\big(\rhod(Y) - \rhodb(Y)\big).$$
As $\DEA$ is a graded derivation, vanishing on smooth functions on $M$, we have that for any $\varphi \in \Omega_{\C}(M)$ 
$$\DEA(\varphi) = -\big(\mu + \mub\big)(\varphi) + \tfrac{t+1}{4}\big(\taud(\varphi) + \taudb(\varphi) - \theta \wedge \varphi \big) +\tfrac{3t-1}{4}i\big(\rhod(\varphi) - \rhodb(\varphi)\big).$$ The result then follows by Proposition \Cref{DA=dA+dA*+Ir(A)} and the identity $r(A) = \tfrac{t+1}{2} \theta$. \end{proof}

\section{Almost Hermitian Identities via the Bismut Dirac Operator}
We have shown above that the canonical Hermitian and Riemannian Dirac operators on $\Gamma
\Cl(M)$ agree up to a tensorial expression. In particular we have by \Cref{gaudmainprop}

$$D_t - \rD \cong -(\mu + \mub + \mu^* + \mub^*) + \tfrac{t+1}{4}(\tau + \taub + \tau^* + \taub^* - E_\theta + I_\theta) + \tfrac{3t-1}{4}i(\rhod - \rhodb -\rhod^* + \rhodb^*).$$

 Since $\nabla^t = \levic$  if and only if $d\omega = 0 = N$ we have  $\rD = D_t$ if and only if the almost Hermitian manifold is K{\"a}hler. 

 We recall that for any Dirac operator $D$ defined in terms of a Hermitian connection we have by Proposition \Cref{[DH]} $$ [D, \Ht] = -i  D_c .$$ Moreover since $D_t$ is a Hermitian Dirac operator defined in terms of a Gauduchon connection we obtain the following

\begin{proposition}\label{[hDH]} Let $M$ be an almost Hermitian manifold. We have the following
identities

\begin{equation*}
\begin{split} [D_t, \Ht] = -i(D_t)_c \hspace{.2in} & \hspace{.2in}
  [(D_t)_c, \Ht] = \phantom{-}iD_t \\
[D_t^\trs, \Ht] = \phantom{-} i(D_t)_c^\trs \hspace{.2in}  &\hspace{.2in}
  [(D_t)_c^\trs, \Ht] = -iD_t^\trs .
\end{split}
\end{equation*} \end{proposition}

Using the expression previously obtained for the adjoint of $D_t$ in \Cref{D^*} we observe that 
$$ \left (D_{t} \right)_c  = \left (D_{t} \right)_c^* + \tfrac{t+1}{4}L_{J\theta} $$
and thus

$$ [D_t, \Ht] = -i \left (D_{t} \right)_c^* - \tfrac{i(t+1)}{4}L_{J\theta}.$$
In particular for any almost Hermitian manifold $M$ we have the Dirac operator $B \stackrel{\textrm{def}}{=} D_{-1}$ is self-adjoint and so
$$ [\bD, \Ht] = -i \bD_c = -i \bD_c^* .$$

\begin{definition} We define the operators $\hd$ and $\hdb$ on $\Omega_\C(M)$ by 
$$\hd = \del- i\rhod \hspace{.1in} \textrm{ and } \hspace{.1in}\hdb = \delb+ i \rhodb .$$
Following \cite{Ta_To} we also define the differential operators $\delta$ and $\deltab$ on $\Omega_\C(M)$ by   $$\delta = \del+ \mub \hspace{.1in} \textrm{ and } \hspace{.1in}\deltab = \delb+ \mu$$ where $d = \delta + \deltab$ is the exterior derivative on $\Omega_\C(M). $\end{definition}

We note that through the canonical isomorphism $\Gamma \Cl(M) \cong \Omega_\C(M)$ we have 

\begin{equation}\label{Dcongetc} \rD = \delta + \deltab + \delta^* + \deltab^* \hspace{.2in}\textrm{ and } \hspace{.2in} B = \hd + \hdb + \hd^* + \hdb^* .  \end{equation} 
The zero-order term $\mu$ vanishes if and only if $M$ is integrable, and symmetrically we show below in \Cref{rho=0iff} that $\rhod$ vanishes if and only if $\del \omega = \delb \omega = 0 .$

 Note that $J \mapsto -J^\vee$ by the metric isomorphism $T(M)
\cong T^\vee(M)$. This has the effect of flippling the sign through the
isomorphism $\Gamma(\Cl(M)) \cong \Omega_{\C}(M)$ when conjugating
by $\Ja$. From this we obtain the

\begin{lemma}\label{DcongetcJ}  On any almost Hermitian manifold we have $$\rD_c \cong
-d_c - d_c^* = i( \delta - \deltab + \deltab^* -\delta^*)$$
and $$ \bD_c \cong i( \hd - \hdb + \hdb^* - \hd^*).$$ \end{lemma}
\begin{proof} The first identity is clear from the above remark. The second identity follows from the observation $\Ja^{-1}\rhod \Ja = -i \rhod$ so that $\hd_c = \del_c -i \rhod{_c} = -i(\del - i\rhod)= -i\hd$ and similarly $\hdb_c = i \hdb$. \end{proof}

The relation $[\bD, \Ht] = -i \bD_c$ implies the following identities of operators on $\Omega(M).$

\begin{proposition}\label{AHid} Let $M$ be an almost Hermitian manifold. We obtain the following
identities of operators on $\Omega_\C(M)$:
\begin{equation*}
\begin{split} [\Lambda, \hdb] = -i \hd^*  \hspace{.2in}&\hspace{.2in} [L,\hd^*]
  = i \hdb\\
  [\Lambda, \hd] = i \hdb^*   \ \ \hspace{.2in} &\hspace{.2in}  [L, \hdb^*] = -i
  \hd  \\
0 = [\Lambda, \hd^*] = [\Lambda, \hdb^*] \hspace{.1in}&\hspace{.25in} [L, \hd] = [L, \hdb] = 0 .
\end{split}
\end{equation*}
\end{proposition}

\begin{proof} The identity $[\bD, \Ht] = -i \bD_c$ implies $$ \left [\hd+ \hdb+ \hd^* + \hdb^*, i(\Lambda - L) \right ] =
  \hd- \hdb+ \hdb^* - \hd^* .$$
That is \small $$ i[\hd, \Lambda] + i[\hdb, \Lambda] + i[\hd^*, \Lambda] +
i[\hdb^*,\Lambda] - i[\hd,L] -i[\hdb,L] -i[\hd^*,L] -i[\hdb^*,L]
= \hd- \hdb+ \hdb^* - \hd^* .$$ \normalsize
The result then follows by comparing bidegree of the
operators in the above equality. For example, $[\hdb^*, L]$ has bidegree $(1,0)$, as does
$\hd$ and so the equality implies $i \hd= [\hdb^*, L ]$.\end{proof}

\begin{corollary}\label{rho=0iff} Let $M$ be an almost Hermitian manifold. Then $\rhod = 0$ if and only if $\del \omega = \delb \omega = 0$.\end{corollary}

\begin{proof} Evidently if $\del \omega = 0$ then $\rhod = 0$, and by conjugating by complex conjugation $\del \omega = 0$ if and only if $\delb \omega = \ba{\del \ba{\omega}} = \ba{\del \omega} =  0$.
Conversely, by the identity $$[L, \hd] = 0$$ we have
$$[L, \del] = i [L, \rhod] .$$ 
Since both $\del$ and $\rhod$ are graded derivations, by evaluating at $1$, we see $i\rhod(\omega) = \del\omega$. Hence if $\rhod = 0$ we have $\del\omega = 0.$\end{proof}

Conjugating our Dirac operators by the transpose map $\trs$ we obtain the following correspondences to operators on forms. For the antipodal involution $\alpha$ on $Cl(M)$ we have $\alpha \nabla = \nabla \alpha$ for any metric connection $\nabla$. In particular $B \alpha = - \alpha B$ and $\rD \alpha = - \alpha \rD$. Evidently the operators $\hd$ and $\delta$ also anticommute with $\alpha$.

\begin{lemma}\label{Dcongetctrs} On any almost Hermitian manifold we have $$\rD^\trs \cong d\alpha -
d^*\alpha = (\delta + \deltab - \delta^* - \deltab^*)\alpha$$ and $$ \bD^\trs \cong (\hd + \hdb - \hd^* -
\hdb^*)\alpha .$$ \end{lemma}

\begin{proof} Let $\varphi \in \Gamma(Cl(M)) \cong \Omega(M)$. Recall that for any $v \in TM$ we have $$v \cdot \varphi \cong
v \wedge \varphi -
v \inter \varphi \textrm{ and } \varphi \cdot v \cong \big(v \wedge
\alpha \varphi+ v \inter 
\alpha \varphi \big).$$
Since $\displaystyle d = \sum_k v_k \wedge \levic_{v_k} \textrm{ and } d^* = -\sum_k
v_k \inter \levic_{v_k}$ on $\Omega(M)$ we see that $$\rD^\trs(\varphi) = \sum_k
\levic_{v_k}(\varphi) \cdot v_k \cong \sum_k
v_k \wedge \levic_{v_k}(\alpha\varphi) + \sum_k
v_k \inter \levic_{v_k}(\alpha\varphi)  = d (\alpha \varphi) - d^*
(\alpha \varphi).$$
Moreover $$\DA^\trs(\varphi) = \sum_j \A_{v_j}(\varphi) \cdot v_j  \cong \sum_j v_j \wedge \A_{v_j}(\alpha \varphi) + \sum_j v_j \inter \A_{v_j}(\alpha \varphi)  $$
and thus $$ \bD^\trs(\varphi) - \rD^\trs(\varphi) \cong \DEA(\alpha \varphi) - \DEA^*(\alpha \varphi).$$
Since for $t=-1$, by the proof of \Cref{gaudmainprop} we have $$\DEA = -\mu - \mub -i\rhod + i \rhodb \hspace{.15in}\textrm{ and hence } \hspace{.15in} \DEA^*= -\mu^* -\mub^* + i \rhod^* - i \rhodb^*,$$
we conclude $\bD^\trs(\varphi) \cong \hd (\alpha\varphi) + \hdb (\alpha\varphi) - \hd^* (\alpha\varphi) -\hdb^* (\alpha\varphi)$ as desired.\end{proof}

 For operators $A,B$ on $\Gamma(\Cl(M)) \cong \Omega_{\C}(M)$ we define the anti-commutator \[ \left \{A, B \right \} = AB + BA .\]

\begin{remark} By the bracket $[\cdot,\cdot]$ of operators on sections of the Clifford or exterior algebra bundles we will always mean the vanilla (ungraded) commutator. Much of the discussion can be rewritten via a $\mathbb{Z}_2$ graded commutator, though there are some inconveniences one is forced to address when adopting this approach.\end{remark}

\begin{proposition}\label{BLtLta} On any almost Hermitian manifold $M$ we have the following identities

\begin{equation*}
    \begin{split}
        \left \{\Lt + \Lta, \bD \right \} &= \bD^\trs \hspace{.5in}  \left \{\Lt -
  \Lta, \bD \right \}  \ = -i \bD_c^\trs \\
\left \{\Lt + \Lta, \bD^\trs \right \}  &= \bD \hspace{.5in} \  \left \{ \Lt -
  \Lta, \bD^\trs \right \}  = \phantom{-} i \bD_c \\
\left \{ \Lt + \Lta, \bD_c \right \}  &= \bD_c^\trs \hspace{.5in}  \left \{ \Lt -
  \Lta, \bD_c \right \}  = \phantom{-}i \bD^\trs \\
    \end{split}
\end{equation*}
\end{proposition}

\begin{proof}  By \Cref{curlyH=i(Lambda-L)} we have through the identification $ \Gamma(\Cl(M)) \cong \Omega_{\C}(M)$ that

\begin{equation*}
    \begin{split} \left \{ \Lt + \Lta, \bD \right \}  &=  \left \{\alpha H, \hd + \hdb + \hd^* + \hdb^* \right \} \\
        &= \left [\hd + \hdb + \hd^* + \hdb^*, H\right] \alpha \\
        &= \left [ \hd + \hdb , H \right ]\alpha + \left [ \hd^* + \hdb^*, H \right ]\alpha \\
        &= (\hd + \hdb - \hd^* - \hdb^*) \alpha \\
        &= \bD^\trs 
        \end{split}
\end{equation*} while by \Cref{AHid} and \Cref{curlyH=i(Lambda-L)} we observe
\begin{equation*}
    \begin{split} \left \{ \Lt - \Lta, \bD \right \}  &=  i \left [ \Lambda + L, \hd + \hdb + \hd^* + \hdb^* \right ] \alpha \\
        &= i\left [ \Lambda + L, \hd + \hdb \right] \alpha + i\left [ \Lambda + L, \hd^* + \hdb^* \right] \alpha\\
        &= (\hd - \hdb + \hd^* - \hdb^*)\alpha \\
        &= -i \bD_c^\trs .
        \end{split}
\end{equation*} 
The remaining
identities are obtained by conjugating with the operators $\Ja$ and $\trs$. By the identities $\Lt_c = \Lt$ and $\Lta_c = \Lta$ we see
\begin{equation*} \left \{\Lt + \Lta, \bD_c \right \} = \Ja^{-1}\left \{\Lt + \Lta, \bD \right \}\Ja = \Ja^{-1}\bD^\trs\Ja = \bD_c^\trs . \qedhere \end{equation*} \end{proof}

 Summing the above identities gives rise to new objects of interest. For example  $$ \left \{\Lt, \bD^\trs \right \} = \tfrac{1}{2} \big(\left \{ \Lt + \Lta, \bD^\trs \right \} +  \left \{\Lt - \Lta, \bD^\trs \right \}
 \big)= \tfrac{1}{2} (\bD + i \bD_c ). $$
We define $$\D = \tfrac{1}{2}(\bD + i\bD_c) \hspace{1in} \cD = \tfrac{1}{2}(\bD - i\bD_c) $$ 
and $$\trD = \tfrac{1}{2}(\rD +
i\rD_c) \hspace{1in} \ba{\trD} = \tfrac{1}{2}(\rD - i\rD_c).$$
More generally we recall the operator $\Dt$ on $\Gamma \big(\Cl(M)\big)$ defined by

$$ \Dt = \tfrac{1}{2} \big( D_t + i (D_t)_c\big). $$

 Notice that for any $\varphi \in \Gamma(\Cl(M))$ and for any $J$-adapted local orthonormal frame \\ $v_1, \dots, v_{2n} = e_1, \dots, e_n, Je_1, \dots, Je_n$ of $TM$ we have 
$$ \Dt(\varphi) = \tfrac{1}{2}\big (\sum_{j=1}^{2n} v_j \cdot \nabla^t_{v_j}(\varphi) - iJv_j \cdot \nabla^t_{v_j}(\varphi) \big) = 2\sum_{j=1}^n \e_j \cdot \nabla^t_{\ba{\e}_j}(\varphi).$$

We observe (cf. \cite{Mi}) that the operator $\Dt$ is of Clifford bidegree $(1,1)$ as left Clifford multiplication by an element of $T^{1,0}(M)$ is of Clifford bidegree $(1,1)$ and any Hermitian connection is both $\Ht$ and $\J$ parallel. Contrarily the operator $\trD$ on sections of $\Cl(M)$ is not generically an operator of pure Clifford bidegree.

Focusing attention once again on our operators $\trD$ and $\D$, we state the below immediate consequence of the identity (\Cref{Dcongetc}) along with Lemmas  \cref{DcongetcJ} and \cref{Dcongetctrs}.

\begin{proposition}\label{CurlyDonforms} On any almost Hermitian manifold $M$ we have the following correspondences of operators through the isomorphism $\Gamma \Cl(M) \cong \Omega_{\C}(M)$ 

\begin{equation*}
    \begin{split} \trD \phantom{i} = \deltab + \delta^* \hspace{.48in} &\hspace{.34in} \ba{\trD}\phantom{i} = \delta + \deltab^*  \\
    \trD^\trs = \deltab \alpha - \delta^* \alpha \hspace{.3in} &\hspace{.3in} \ba{\trD}^\trs = \delta \alpha - \deltab^* \alpha\\
    \D \phantom{i}= \hdb+ \hd^* \hspace{.48in} &\hspace{.32in}\cD \phantom{i} = \hd+ \hdb^* \\
    \D^\trs = \hdb\alpha - \hd^*\alpha
     \hspace{.3in} &\hspace{.3in}\cD^\trs = \hd\alpha - \hdb^*\alpha \\
    \end{split}
\end{equation*}
\end{proposition}

\begin{proof} We observe that by the identity (\ref{Dcongetc}) we have on $\Gamma \Cl(M) \cong \Omega_{\C}(M)$ that $$\rD = \delta + \deltab + \delta^* + \deltab^*$$ and by \Cref{DcongetcJ} we have $\rD_c = i(\delta - \deltab +\deltab^* - \delta^*) $ so that $$ \trD = \tfrac{1}{2}(D + iD_c) = \deltab + \delta^* .$$ The other identities are similarly obtained by use of (\cref{Dcongetc}) along with Lemmas \cref{DcongetcJ} and \cref{Dcongetctrs}.\end{proof}

\begin{lemma}\label{[rDrDtrs]=0} Let $M$ be an almost Hermitian manifold. Then we have $$[\rD,\rD^\trs] =
[\rD_c, \rD_c^\trs] = 
0 .$$ Furthermore if $M$ is K{\"a}hler then
$$[\bD,\bD^\trs] = [\bD_c, \bD_c^\trs] = 0 .$$ \end{lemma}

\begin{proof} We observe $$ \rD \rD^\trs \cong (d + d^*)(d\alpha - d^*\alpha) = d^*d\alpha - dd^*\alpha =
(d\alpha - d^*\alpha)(d + d^*) \cong \rD^\trs \rD .$$
It follows that $[\rD, \rD^\trs] = 0$. Conjugating this identity by
$\Ja$ gives $[\rD_c, \rD_c^\trs] = 0$. 

 If $M$ is K{\"a}hler then $\bD = \rD$ and hence \begin{equation*} [\bD,\bD^\trs] = [\bD_c, \bD_c^\trs] = 0 . \qedhere \end{equation*} \end{proof}

\begin{definition} For an operator $T$, we define the \textit{Laplacian of $T$} by $$\Delta_T = TT^* + T^*T.$$
\end{definition}

\begin{proposition}\label{DeltatrD=DeltatrDtrs} On any almost Hermitian manifold $M$ we have $\Delta_{\trD} = \Delta_{\trD^\trs}$. Furthermore $\Delta_{\trD} = \Delta_{\delta} + \Delta_{\deltab}$ through the isomorphism $\Gamma(\Cl(M)) \cong \Omega_{\C}(M)$. \end{proposition}

\begin{proof} The identity $\Delta_{\trD} = \Delta_{\trD^\trs}$ can be checked by observing
\begin{equation*}
    \begin{split} \Delta_{\trD} &= (\deltab + \delta^*)(\delta + \deltab^*) + (\delta + \deltab^*)(\deltab + \delta^*) \\
    &= \Delta_{\delta} + \Delta_{\deltab} + \left \{\delta, \deltab \right \} + \left \{\delta^*, \deltab^* \right \} \\
    \end{split}
\end{equation*}
while 
\begin{equation*}
    \begin{split} \Delta_{\trD^\trs} &= (\deltab \alpha - \delta^* \alpha)(\delta \alpha - \deltab^*\alpha) + (\delta \alpha - \deltab^* \alpha)(\deltab \alpha - \delta^*\alpha) \\
    &= \Delta_{\delta} + \Delta_{\deltab} - \left \{\delta, \deltab \right \} - \left \{\delta^*, \deltab^* \right \} .\\
    \end{split}
\end{equation*}
By \Cref{d^2} we see that $$\left \{\delta, \deltab \right \} = \left \{ \del, \delb \right \} + \left\{ \mu, \mub \right \} + \left \{  \del, \mu \right \} + \left \{ \delb, \mub \right \} = \left \{ \del, \delb \right \} - \left \{ \del, \delb \right \} + 0 + 0 = 0 $$ and so $\Delta_{\trD} = \Delta_{\trD^\trs}$. \end{proof}

 Repeating the above argument for $\Delta_{\D}$ and $\Delta_{\D^\trs}$ we subsequently obtain the

\begin{proposition}\label{Delta=DeltaAH} On any almost Hermitian manifold $M$ we have $\Delta_{\D} + \Delta_{\D^\trs} \cong 2 \left (\Delta_{\hd} + \Delta_{\hdb} \right )$ through the isomorphism $\Gamma(\Cl(M)) \cong \Omega_{\C}(M)$. Furthermore $\Delta_{\D} = \Delta_{\D^\trs}$ if $M$ is K{\"a}hler.\end{proposition}

\begin{proof} By the identities in \Cref{CurlyDonforms} we compute 
\begin{equation*}
    \begin{split} \Delta_{\D} \phantom{i} &= \Delta_{\hd} + \Delta_{\hdb} + \left \{\hd, \hdb\right \} + \left \{\hd^*, \hdb^* \right \}  \textrm{ and }\\
    \Delta_{\D^\trs} &= \Delta_{\hd} + \Delta_{\hdb} - \left \{\hd, \hdb\right \} - \left \{\hd^*, \hdb^* \right \}.  
    \end{split}
\end{equation*}
Hence $\Delta_{\D} - \Delta_{\D^\trs} = 0$ when $\hd = \del$ as then $\left \{\hd, \hdb\right \} = \left \{ \del - i \rhod, \delb + i \rhodb \right \} = \left \{ \del , \delb  \right \} = 0$.\end{proof}

We recall the operator of complex conjugation $c: \Cl^{r,s}(M) \rightarrow \Cl^{-r,-s}(M)$ (see \cite{Mi}). As $\rD$ and $\bD$ are real operators, they commute with $c$. Thus we have the following

\begin{lemma}\label{cCurlyDc} On an almost Hermitian manifold $M$ we have $$c \ \trD c = \ba{\trD}, \hspace{.3in} c \ \D c = \ba{\D}, \hspace{.3in} c \ \D^\trs c = \ba{\D^\trs}. $$  \end{lemma}

\begin{proof} As $\trs$ and $\Ja$ are also real operators, both commute with $c$. Hence $$ c \ \trD c = \tfrac{1}{2} \big ( c \rD c - i c \Ja^{-1}\rD \Ja c\big ) = \tfrac{1}{2} \big( \rD - i\rD_c \big) = \ba{\trD}.$$ The remaining identities are obtained similarly. \end{proof}

We now turn to investigating the relationship between the intrinsic $sl(2)$ operators on $\Cl(M)$ with our operators $\D$ and $\trD$.

\begin{proposition}\label{HcurlyBetc} 

On any almost Hermitian
manifold $M$ the following identities obtain
\begin{equation*}
\begin{split} \left \{ \D , \Lt \right \}  = 0 \hspace{.5in} & \hspace{.5in} \left \{ \cD , \Lta \right \}= 0 \\
\left \{ \D , \Lta \right \} = \D^\trs \hspace{.4in}& 
 \hspace{.5in} \left \{ \cD,  \Lt \right \}  = \cD^\trs \\
\hspace{.5in} [\Ht,& \D] = \D.
\end{split}
\end{equation*}
\end{proposition}

\begin{proof} We observe that by \Cref{[hDH]}, 
\begin{equation*}
\begin{split}[\Ht, \D] &= [ \Ht, \tfrac{1}{2}(\bD + i \bD_c)] = \tfrac{1}{2}([ \Ht, \bD] +i [\Ht, \bD_c]) \\
&= \tfrac{1}{2}(\bD + i \bD_c) = \D
\end{split}
\end{equation*}
giving the last identity. For the first identity we observe by \Cref{BLtLta}
\begin{equation*}
\begin{split} \left \{ \D, \Lt \right \} &= \tfrac{1}{2} \big ( (\bD + i
  \bD_c) \Lt + \Lt (\bD + i\bD_c) \big ) = \tfrac{1}{2} \big ( \left \{
  \bD, \Lt \right \} + i \left \{ \bD_c, \Lt \right \}\big ) \\
&= \tfrac{1}{4} \big ( (\bD^\trs - i \bD_c^\trs) + i(\bD_c^\trs + i
\bD^\trs ) \big ) = 0.
\end{split}
\end{equation*}
A similar argument, again using \Cref{BLtLta} yields $\left \{ \D, \Lta \right \} = \D^\trs$. The remaining identities are obtained by conjugating by complex conjugation. For example by \Cref{cCurlyDc} 
we have
\begin{equation*} \left \{ \cD , \Lt \right \} = \left \{ c \ \D c , c \  \Lta  c \right \} = c \left \{ \D , \Lta \right \} c =  c \ \D^\trs c = \cD^\trs . \qedhere \end{equation*} \end{proof}

\begin{proposition}\label{[HDeltaB]etc} On any almost Hermitian manifold $M$ we have

\begin{equation*}
    \begin{split}[ \Ht , \Delta_{\D} ] &= 0 \hspace{.68in} [\Ht, \Delta_{\D^\trs}] = 0 \\
        [\Lt, \Delta_{\D}] &= [\cD^\trs, \D] \hspace{.3in} [\Lt, \Delta_{\D^\trs}] = [\D, \cD^\trs] \\
        [\Lta, \Delta_{\D}] &= [\D^\trs, \cD] \hspace{.3in} [\Lta, \Delta_{\D^\trs}] = [\cD, \D^\trs] .
    \end{split}
\end{equation*}
\end{proposition}

\begin{proof} We compute by \Cref{HcurlyBetc}
\begin{equation*}
\begin{split} [ \Ht, \Delta_{\D} ] &= \Ht\D\cD + \Ht\cD\D - \D\cD\Ht -
  \cD\D\Ht + \D\Ht\cD - \D\Ht\cD \\
&=(\Ht\D - \D\Ht)\cD + \D(\Ht\cD - \cD\Ht) + \Ht\cD\D - \cD\D\Ht \\
&= \Ht\cD\D - \cD\D\Ht \\
&= \Ht\cD\D - \cD\D\Ht + \cD\Ht\D - \cD\Ht\D\\
&= (\Ht\cD - \cD\Ht)\D +\cD(\Ht\D - \D\Ht) = 0.
\end{split}
\end{equation*}
Similarly
\begin{equation*}
    \begin{split} [\Lt, \Delta_{\D}] &= \Lt \D \cD + \Lt \cD \D - \D \cD \Lt - \cD \D \Lt \\
    &= \Lt \D \cD + \Lt \cD \D - \D \cD \Lt - \cD \D \Lt + \D \Lt \cD - \D \Lt \cD \\
    &= \big( (\Lt \D + \D \Lt)\cD - \D(\cD\Lt + \Lt \cD) + \Lt\cD\D -\cD\D\Lt \big)\\
    &=\big( -\D\cD^\trs + (\Lt\cD + \cD \Lt)\D - \cD(\D\Lt + \Lt\D) \big) \\
    &= [\cD^\trs, \D].
    \end{split}
\end{equation*}
The remaining identities are obtained by conjugating with $\trs$ and complex conjugation. \end{proof}

\section{Almost K{\"a}hler identities via the Riemannian Dirac Operator}

\begin{lemma} Let $M$ be an almost K{\"a}hler manifold.  For any $v \in TM$ and any $\varphi \in \Gamma \Cl(M)$ we have $$\left [\levic_v, \J \right]\varphi  = i\left ( \Ja^{-1}\levic_{Jv} (\Ja \varphi) - \levic_{Jv}\varphi \right ). $$ \end{lemma}

\begin{proof} We recall that $\J = -i \Jd$ on $\Cl(M)$. As the commutator of two derivations is a derivation, and conjugating a derivation by an algebra automorphism is a derivation we observe $[\levic_v, \Jd] + \Ja^{-1}\levic_{Jv}\circ \Ja$ is a derivation. Thus, it suffices to check the identity 
$$[\levic_v, \Jd] + \Ja^{-1}\levic_{Jv}\circ \Ja = \levic_{Jv} $$ 
on functions and vector fields. Noticing that $\Jd$ vanishes on functions, and that $\Ja$ is the identity map on functions, we see that for a smooth function $f$ on $M$ we have
$$[\levic_v, \Jd](f) + \Ja^{-1}\levic_{Jv}(\Ja f) = \Ja^{-1}\levic_{Jv}(\Ja f) = \levic_{Jv}(f) . $$ If $X \in
  \Gamma(TM)$ then
$$ [\levic_v, \Jd](X) + \Ja^{-1}\levic_{Jv}(\Ja X) = \levic_{v}JX - J\levic_v X + J^{-1}\levic_{Jv}JX = \levic_v JX -
J\levic_vX - J\levic_{Jv}JX .$$
When $d\omega = 0$ we have (see, for instance, \cite{Kob_Nom}) the identity  $(\levic_v J) = J(\levic_{Jv} J)$. That is $$\levic_v JX -J \levic_v X = J \big(\levic_{Jv}JX - J \levic_{Jv}X \big) $$ and hence $$[\levic_v, \Jd](X) + \Ja^{-1}\levic_{Jv}(\Ja X)= \levic_vJX - J\levic_vX - J\levic_{Jv}JX = - J^2\levic_{Jv}X = \levic_{Jv}X .$$ The result follows by substituting the identity $\J = -i\Jd$ on the complex Clifford bundle $\Cl(M).$ \end{proof}

\begin{proposition}\label{AKDI} On any almost K{\"a}hler manifold $M$ we have the following
identities

\begin{equation*}
\begin{split} [\rD, \Ht] = -i\rD_c \hspace{.2in} & \hspace{.2in}
  [\rD_c, \Ht] = \phantom{-}i\rD \\
[\rD^\trs, \Ht] = \phantom{-} i\rD_c^\trs \hspace{.2in}  &\hspace{.2in}
  [\rD_c^\trs, \Ht] = -i\rD^\trs .
\end{split}
\end{equation*}
\end{proposition}

\begin{proof} Let $v_1,\dots,v_{2n}$ be a $J$-adapted orthonormal frame of $TM$. We first compute $[\rD, \J]$. For any $\varphi \in \Gamma(\Cl(M))$ we have
\begin{equation*}
\begin{split} [\rD, \J](\varphi) &= \sum_j \big(v_j\cdot \levic_{v_j}\J \varphi - \J(v_j\cdot \levic_{v_j}\varphi) \big) \\
&= \sum_j \big(v_j\cdot \levic_{v_j}\J \varphi + i Jv_j \cdot \levic_{v_j}\varphi - v_j \cdot \J(\levic_{v_j}\varphi) \big) \\
&=\sum_j \big( v_j \cdot \left[ \levic_{v_j},\J \right ](\varphi) + iJv_j \cdot \levic_{v_j}\varphi\big) .\\
\end{split}
\end{equation*}
By the identity $\left [\levic_v, \J \right]\varphi  = i\left ( \Ja^{-1}\levic_{Jv}\Ja \varphi - \levic_{Jv}\varphi \right ) $ above, we have

\begin{equation*}
\begin{split}
[\rD, \J] (\varphi) &= i \sum_j \big( v_j \cdot \Ja^{-1}\levic_{Jv_j}\Ja\varphi - v_j\cdot\levic_{Jv_j}\varphi + Jv_j \cdot \levic_{v_j}\varphi\big)\\
&= i\rD_c(\varphi) - 2i \sum_jv_j\cdot \levic_{Jv_j}\varphi .
\end{split}
\end{equation*}
Let $\tD(\varphi) =  \sum_jv_j\cdot \levic_{Jv_j}\varphi$ and let $L_{\om} $ be left-multiplication by $\om$. Clearly $\Ht + \J = 2L_{\om}$ and $[\rD, \J] = i\rD_c -2i \tD$. Furthermore, for any $\varphi \in \Gamma(\Cl(M))$

\begin{equation*}
    \begin{split} \rD(\om \cdot \varphi) &= \sum_j\big(v_j\cdot (\levic_{v_j}\om)\cdot \varphi + v_j\cdot \om\cdot \levic_{v_j}\varphi \big)\\
    &=\rD(\om)\cdot \varphi + \sum_j\big ( ( \om\cdot v_j + iJv_j) \cdot \levic_{v_j}(\varphi)\big ) \\
    &= \om \cdot \rD(\varphi) - i\tD(\varphi).
    \end{split}
\end{equation*}
Thus we have that $[\rD, L_{\om} ] = -i \tD$. But then
$$[\rD, \Ht] + [\rD, \J] = 2[\rD, L_{\om}] =-2i \tD$$
and so $ [ \rD, \Ht ] = -i \rD_c $. The remaining identities are
obtained by conjugating with $\Ja$ and $\trs$ as in the proof of \Cref{[DH]}. \end{proof}

\begin{lemma}\label{[rD_c,rDtrs]etc} Let $M$ be an almost  K{\"a}hler manifold. Then $ [\rD_c, \rD^\trs] = [\rD, \rD_c^\trs]$. \end{lemma}

\begin{proof} By the identites in \Cref{AKDI} we have $$ [\rD_c, \rD^\trs] = i[ [\rD, \Ht], \rD^\trs] $$
and $$ [\rD, \rD_c^\trs] = -i [\rD, [\rD^\trs, \Ht]] .$$ By use of \Cref{[rDrDtrs]=0} it is quickly checked that the difference \begin{equation*} [\rD_c, \rD^\trs] - [\rD, \rD_c^\trs] = i [ [\rD, \rD^\trs], \Ht] = 0 . \qedhere \end{equation*} \end{proof}

 As a consequence of \Cref{AKDI} we recover the \textit{almost K{\"a}hler identities}. (See \cite{Ci_Wi}, \cite{Ta_To}).

\begin{proposition}\label{AK_id} Let $M$ be an almost K{\"a}hler manifold. We have the following
generalizations of the K{\"a}hler identities: 
\begin{equation*}
\begin{split} [\Lambda, \deltab] = -i \delta^*  \hspace{.2in}&\hspace{.2in} [L,\delta^*]
  = i \deltab\\
  [\Lambda, \delta] = i \deltab^*   \ \ \hspace{.2in} &\hspace{.2in}  [L, \deltab^*] = -i
  \delta  \\
0 = [\Lambda, \delta^*] = [\Lambda, \deltab^*] \hspace{.1in}&\hspace{.25in} [L, \delta] = [L, \deltab] = 0 .
\end{split}
\end{equation*}
\end{proposition}

\begin{proof} By the identity $[\rD, \Ht] = -i\rD_c$ we have that $ [d + d^*, i(\Lambda -L)] = i(d_c + d_c^*) $ or equivalently $$ i[d, \Lambda] + i[d^*, \Lambda] -i [d, L] - i [d^*, L] = id_c + id_c^*.$$ 
By considering degree of each operator in the above equality we obtain $$[d, \Lambda] = d_c^* \hspace{.2in}  [d^*, L] = -d_c \hspace{.2in}\textrm{ and } \hspace{.1in} [d, L] = [d^*,\Lambda] = 0 .$$ Furthermore since $$[\Lambda, d] = [\Lambda, \delta + \deltab] = [\Lambda, \del + \mub + \delb+ \mu] = [\Lambda, \del] + [\Lambda,\mub] + [\Lambda, \delb] + [\Lambda, \mu]$$
the above implies $[\Lambda,d] = -id_c^* = i(\deltab^* - \delta^* ) = i( \delb^*  + \mu^*- \del^* - \mub^*)$ as well. The argument is again completed by considering the effect of each operator in the equality on bidegree. \end{proof}

It is quick to check, using the identities of \Cref{AK_id} that $\Delta_\delta = \Delta_{\deltab} $ in the almost K{\"a}hler setting. (See, for instance, \cite[Proposition 6.2]{Ta_To}). Hence by \Cref{DeltatrD=DeltatrDtrs} we have

\begin{corollary}\label{DeltaD=Deltadelta} For an almost K{\"a}hler manifold $M$ we have through the identification $\Gamma(\Cl(M)) \cong \Omega_{\C}(M)$ that $$\Delta_\trD = 2\Delta_{\delta}.$$ \end{corollary}

\begin{proposition}\label{rDcurlyL} On any almost K{\"a}hler manifold $M$ we have the following identities

\begin{equation*}
\begin{split} 
\left \{ \Lt + \Lta, \rD \right \}  &= \rD^\trs \hspace{.5in}  \left \{ \Lt -
  \Lta, \rD \right \}  \ = -i \rD_c^\trs \\
\left \{ \Lt + \Lta, \rD^\trs \right \}  &= \rD \hspace{.5in} \  \left \{ \Lt -
  \Lta, \rD^\trs \right \}  = \phantom{-} i \rD_c \\
\left \{ \Lt + \Lta, \rD_c \right \}  &= \rD_c^\trs \hspace{.5in}  \left \{ \Lt -
  \Lta, \rD_c \right \}  = \phantom{-}i \rD^\trs .\\
\end{split}
\end{equation*}
\end{proposition}

\begin{proof} We have through the identification $ \Gamma(\Cl(M)) \cong \Omega_{\C}(M)$ that
\begin{equation*}
\begin{split} \left \{ \Lt + \Lta , \rD \right \} &= \alpha H (d + d^*)
  + (d + d^*) \alpha
  H \\
&= \alpha [ H, d + d^*] = \alpha([H, d] + [H, d^*]) \\
&= d \alpha - d^* \alpha \\
&= \rD^\trs 
\end{split}
\end{equation*}
and
\begin{equation*}
\begin{split} \left \{ \Lt - \Lta , \rD \right \} & = -i \left ( \alpha
  (\Lambda + L) (d + d^*) + (d + d^*) \alpha (\Lambda + L) \right ) = -i \alpha [
  (\Lambda + L), d + d^* ] \\
&= -i \alpha [\Lambda, d ] - i \alpha [L, d^*] = i \alpha d_c^* - i
\alpha d_c \\
&= -i (d_c^* \alpha - d_c \alpha).
\end{split}
\end{equation*}
Hence $ \left \{ \Lt - \Lta , \rD \right \} = -i \rD_c^\trs
$. The remaining
identities are again obtained by conjugating with the operators $\Ja$ and $\trs .$ \end{proof}

\begin{proposition}\label{[HcurlyD]etc} On any almost K{\"a}hler
manifold $M$ the following identities obtain

\begin{equation*}
\begin{split}
\left \{ \trD , \Lt \right \} = 0 \hspace{.5in} & \hspace{.5in} \left \{ \ba{\trD} ,
  \Lta \right \} = 0\\
\left \{ \trD , \Lta \right \} = \trD^\trs  \hspace{.4in}& 
 \hspace{.5in} \left \{
  \ba{\trD} ,  \Lt \right \}  = \ba{\trD}^\trs \\
\hspace{.5in} [\Ht,& \trD] = \trD .
\end{split}
\end{equation*}
\end{proposition}

\begin{proof} We observe that by \Cref{rDcurlyL}
\begin{equation*}
\begin{split} \left \{ \trD, \Lt \right \} &= \tfrac{1}{2} \big ( (\rD + i
  \rD_c) \Lt + \Lt (\rD + i\rD_c) \big ) = \tfrac{1}{2} \big ( \left \{
  \rD, \Lt \right \} + i \left \{ \rD_c, \Lt \right \}\big ) \\
&= \tfrac{1}{4} \big ( (\rD^\trs - i \rD_c^\trs) + i(\rD_c^\trs + i
\rD^\trs )\big ) = 0 .
\end{split}
\end{equation*}
The argument $\left \{ \trD , \Lta \right \} = \trD^\trs $ follows similarly. By \Cref{AKDI}
\begin{equation*}
\begin{split}[\Ht, \trD] &= [ \Ht, \tfrac{1}{2}(\rD + i \rD_c)] = \tfrac{1}{2}([ \Ht, \rD] +i [\Ht, \rD_c]) \\
&= \tfrac{1}{2}(\rD + i \rD_c) = \trD.
\end{split}
\end{equation*}
The remaining identities follow by conjugating by complex conjugation and \Cref{cCurlyDc}. \end{proof}

\begin{proposition}\label{[HDeltaD]etc} On any almost K{\"a}hler manifold $M$ we have

$$[ \Ht , \Delta_{\trD} ] = 0, \hspace{.25in}
        [\Lt, \Delta_{\trD}] = 0  \hspace{.1in} \textrm{ and } \hspace{.1in}  [\Lta, \Delta_{\trD}] = 0 .$$

\end{proposition}

\begin{proof}  Using \Cref{[HcurlyD]etc} and an identical argument to the proof of \Cref{[HDeltaB]etc} we obtain $$[\Ht, \Delta_{\trD}] = 0 \textrm{ and } [\Lt, \Delta_{\trD}] = [\ba{\trD}^\trs, \trD].$$ 
By definition of $\trD$ and $\ba{\trD}^\trs$ we see that $$[\ba{\trD}^\trs, \trD] = \tfrac{1}{4} \left ( [\rD^\trs, \rD] + [\rD_c^\trs, \rD_c ] -i [\rD_c^\trs, \rD] +i [\rD^\trs, \rD_c] \right ) $$ and so the result follows by Lemmas \ref{[rDrDtrs]=0} and \ref{[rD_c,rDtrs]etc}. That is  $$[\rD,\rD^\trs] = 0 \hspace{.2in} \textrm{ and } \hspace{.2in} [\rD_c,\rD^\trs] - [\rD_c^\trs,\rD] = 0 .$$
\end{proof}

\section{Clifford Harmonics}

Let $M$ be a compact almost Hermitian manifold. For any finite
collection of operators $T_j$, for $j = 1, \dots, n$,
on $\Gamma \Cl(M) \cong \Omega_{\C}(M)$ we have $$\displaystyle \varphi \in \ker
\left(\sum_{j=1}^n\Delta_{T_j}\right ) \hspace{.05in}\textrm{ if any
only if } \hspace{.1in} 0 = \displaystyle  \sum_{j=1}^n|| T_j \varphi ||^2 + ||
T_j^* \varphi ||^2  . \hspace{1.6in} $$ I.e. $\displaystyle \varphi \in \ker
\left(\sum_{j=1}^n\Delta_{T_j}\right )$ if and only if 
$\displaystyle \varphi \in \bigcap_{j=1}^n\ker(T_j) \cap \ker(T_j^*)$.

\begin{definition} For an elliptic differential operator $T$ on $\mathcal{A} = \Gamma \Cl(M), \hspace{.01in} \Omega_{\C}(M) $ we define
$$ \har_T = \ker(\Delta_T). $$
Furthermore we define $$  \har_T^{r,s} = \har_T \cap \mathcal{A}^{r,s}. $$ \end{definition}

Recall from \Cref{g:hodgeaut} the Hodge automorphism  $g = \ex(\tfrac{-\pi i }{\phantom{-}4}H) \ex(\tfrac{\pi}{4}(\Lambda-L))$. By \Cref{curlyH=i(Lambda-L)} we rewrite  $$g = \ex\big(-\tfrac{\pi i}{4}\alpha(\Lt + \Lta)\big) \ex\big(-\tfrac{\pi i}{4} \Ht \big) .$$

\begin{proposition} Let $M$ be a compact almost Hermitian manifold. The action of the Hodge automorphism $g$ on $\Gamma(\Cl(M))$ preserves 
$\har_{\D} \cap \har_{\D^\trs}$. \end{proposition}

\begin{proof} By \Cref{[HDeltaB]etc} we have 
\begin{equation*}
\begin{split} 
 [ \alpha( \Lt + \Lta),
\Delta_{\D}] = \alpha [\D^{\trs}, \cD] + \alpha [\cD^{\trs},\D],
\hspace{.2in}&\hspace{.2in} [ \alpha( \Lt + \Lta),
\Delta_{\D^\trs}] = -\alpha [\D^{\trs}, \cD] - \alpha [\cD^{\trs},\D]  \\
 \textrm{ and }  \hspace{.15in} [\Ht, \Delta_{\D}] &=0.\\
\end{split}
\end{equation*}
The first two equalities imply $ [ \alpha( \Lt + \Lta),
\Delta_{\D} + \Delta_{\D^\trs}] = 0$. Since $\Delta_{\D} + \Delta_{\D^\trs}$ commutes with $\alpha( \Lt + \Lta)$ and $\Ht$, and as $g$ is the product of the exponentials of these operators, we have 
$$ g \big(\Delta_{\D} + \Delta_{\D^\trs}\big)  = \big(\Delta_{\D} + \Delta_{\D^\trs} \big)g .$$ This gives the result. \end{proof}

\begin{proposition}\label{hodge_harmonic} Let $M$ be a compact almost K{\"a}hler manifold. The action of the Hodge automorphism $g$ on $\Gamma(\Cl(M))$ preserves $\har_{\trD}$. \end{proposition}

\begin{proof} By \Cref{[HDeltaD]etc} we have $$  [ \alpha( \Lt + \Lta),
\Delta_{\trD}]  = 0 \hspace{.3in} \textrm{ and }  \hspace{.15in}   [\Ht, \Delta_{\trD}] =0.  \hspace{2.2in}$$
Since $\Delta_{\trD}$ commutes with $\alpha( \Lt + \Lta)$ and $\Ht$, and as $g$ is the product of the exponentials of these operators, we have 
$$ g (\Delta_{\trD}) = (\Delta_{\trD})g.$$ This again gives the result. \end{proof}

\begin{proposition}\label{ciso} Let $M$ be a compact almost Hermitian manifold. Complex conjugation $c$ on $\Cl(M)$ induces isomorphisms
$$ c: \har_{\D}^{r,s} \cap \har_{\D^\trs}^{r,s} \rightarrow \har_{\D}^{-r,-s}\cap \har_{\D^\trs}^{-r,-s} ,$$ 
$$ c: \har_{\trD}^{r,s} \longrightarrow \har_{\trD}^{-r,-s}. \hspace{.97in}$$ \end{proposition}

\begin{proof} We recall that by \Cref{cCurlyDc} $$c \ \trD c = \ba{\trD}, \hspace{.3in} c \ \D c = \ba{\D}, \hspace{.3in} c \ \D^\trs c = \ba{\D^\trs}. $$ Hence $c\ \big(\Delta_{\trD}\big)c = \Delta_{\trD}$ and $c\ \big(\Delta_{\D} + \Delta_{\D^\trs}\big)c = \Delta_{\D} + \Delta_{\D^\trs}$. That is $$c: \Gamma(\Cl^{r,s}(M)) \rightarrow \Gamma(\Cl^{-r,-s}(M))$$ descends to complex antilinear isomorphisms \begin{equation*} c: \har_{\trD}^{r,s} \rightarrow \har_{\trD}^{-r,-s}  \textrm{ and } \hspace{.1in} c: \har_{\D}^{r,s} \cap \har_{\D^\trs}^{r,s} \rightarrow \har_{\D}^{-r,-s}\cap \har_{\D^\trs}^{-r,-s}. \qedhere \end{equation*} \end{proof}

\begin{proposition}\label{trsiso} Let $M$ be a compact almost Hermitian manifold. The transpose map $\trs$ induces isomorphisms
$$ \trs: \har_{\D}^{r,s} \cap \har_{\D^\trs}^{r,s} \rightarrow \har_{\D}^{r,-s}\cap \har_{\D^\trs}^{r,-s}, $$
$$ \trs: \har_{\trD}^{r,s} \longrightarrow \har_{\trD}^{r,-s}. \hspace{.9in}$$ \end{proposition}

\begin{proof} We observe $\trs \big(\Delta_{\trD} \big) \trs = \Delta_{\trD^{\trs}}$ and that $\trs \big(\Delta_{\D} + \Delta_{\D^{\trs}}   \big) \trs = \Delta_{\D} + \Delta_{\D^{\trs}}$. By \Cref{DeltatrD=DeltatrDtrs} we have $\Delta_{\trD} = \Delta_{\trD^{\trs}} $ and so the complex algebra anti-automorphism
$$ \trs: \Gamma( \Cl^{r,s}) \rightarrow \Gamma(\Cl^{r,-s}) $$
descends to complex linear isomorphisms

\begin{equation*} \trs: \har_{\trD}^{r,s} \rightarrow \har_{\trD}^{r,-s} \textrm{ and }\hspace{.1in} \trs: \har_{\D}^{r,s} \cap \har_{\D^\trs}^{r,s}  \rightarrow \har_{\D}^{r,-s} \cap \har_{\D^\trs}^{r,-s} . \qedhere \end{equation*} \end{proof}

\begin{theorem}\label{AHsl2har} Let $M$ be a compact almost Hermitian manifold. The Lie algebra generated by $\big ( \Ht, \Lt, \Lta \big)$ on $\Gamma \Cl(M)$ defines a finite dimensional $sl(2)$ representation on the space $\har_{\D} \cap \har_{\D^\trs}.$ \end{theorem}

\begin{proof} By \Cref{[HDeltaB]etc} we have $$[\Lt, \Delta_{\D} + \Delta_{\D^\trs}] = 0 , \hspace{.1in} [\Lta, \Delta_{\D} + \Delta_{\D^\trs}] = 0, \hspace{.1in} \textrm{ and } [\Ht, \Delta_{\D} + \Delta_{\D^\trs}] = 0 .$$ Hence each of the operators $\Lt, \Lta, \Ht$ preserve $\har_{\D}\cap \har_{\D^\trs}$. \end{proof}

 \begin{theorem} \label{AHcor} Let $M$ be a compact almost Hermitian manifold. Through the isomorphism $$\Gamma \Cl(M) \cong \Omega_\C(M)$$ we have
 $$ \har_{\D} \cap \har_{\D^\trs}  \cong \har_{\hd} \cap \har_{\hdb}.$$ \end{theorem}
 
 \begin{proof} By \Cref{Delta=DeltaAH} we have $ 2\big (\Delta_{\hd} + \Delta_{\hdb} \big)=\Delta_{\D} + \Delta_{\D^\trs} $. \end{proof}

\begin{theorem}\label{AKsl2har} Let $M$ be a compact almost K{\"a}hler manifold. The Lie algebra generated by $\big ( \Ht, \Lt, \Lta \big)$ on $\Gamma \Cl(M)$ defines a finite dimensional $sl(2)$ representation on the space $\har_\trD$. \end{theorem}

 \begin{proof} By \Cref{[HDeltaD]etc} we have $$[\Lt, \Delta_{\trD}] = 0 , \hspace{.1in} [\Lta, \Delta_{\trD}] = 0, \hspace{.1in} \textrm{ and } [\Ht, \Delta_{\trD}] = 0 .$$ Hence $\Lt, \Lta, \Ht$ preserve $\har_{\trD}.$ \end{proof}

\begin{theorem}\label{AKcor} Let $M$ be a compact almost K{\"a}hler manifold. Through the isomorphism  $$\Gamma \Cl(M) \cong \Omega_\C(M)$$ we have $$\har_\trD \cong \har_\delta  .$$ \end{theorem}

\begin{proof} By \Cref{DeltaD=Deltadelta} we have $2\Delta_\delta =\Delta_{\trD}$. \end{proof}

\begin{corollary}\label{Hodegehar} On any compact almost Hermitian manifold $M$ we have that the Hodge automorphism induces an isomorphism $$ \har_{\D}^{q-p,n-p-q} \cap \har_{\D^\trs}^{q-p,n-p-q} \cong \har_{\hd}^{p,q}\cap \har_{\hdb}^{p,q} .$$ Moreover, on any compact almost K{\"a}hler manifold $M$ we have that $$\har_{\trD}^{q-p, n - p -q} \cong \har_{\delta}^{p ,q}. $$ \end{corollary}

\begin{proof} By \Cref{g:hodgeaut} we have $$\Gamma(\Cl^{q-p,n-p-q}(M)) \stackrel{g}{\cong} \Omega_\C^{p,q}(M). $$ The result then follows from \Cref{hodge_harmonic} and the previous two theorems. \end{proof}

\begin{example}

Let $M$ be the Kodaira-Thurston manifold. Let $x,y,z,w$ be an orthonormal left-invariant coframe on $M$ with $dz = x\wedge y$ and $dx = dy = dw = 0$. We define an almost complex structure $J$ by $Jy = x$ and $Jz = w$. We compute $d \omega = x \wedge y \wedge w \neq 0$ and $N_J = 0$ so that this almost complex structure is integrable and not an almost K{\"a}hler structure on $M$.

 Recalling the operators $ \hd = \del- i\rhod \hspace{.05in} \textrm{,} \hspace{.05in}\hdb = \delb+ i \rhodb \hspace{.05in}$ and the spaces $$\har_{\hd}^{p,q}(M) \cap \har_{\hdb}^{p,q}(M) = \ker(\hd)\cap\ker(\hdb)\cap\ker(\hd^*)\cap\ker(\hdb^*)\cap \Omega_{\C}^{p,q}(M),$$ we compute the ``Hodge diamond'' of their dimensions to be:

\footnotesize
\[\begin{tikzcd}
	&& \mathbf{1} \\
	& \mathbf{1} && \mathbf{1} \\
	\mathbf{0} && \mathbf{2} && \mathbf{0} \\
	& \mathbf{1} && \mathbf{1} \\
	&& \mathbf{1}
\end{tikzcd}\]
\normalsize

\end{example}

\section*{Appendix: Hermitian Connections and Dirac Operators}

For the convenience of the reader we give here an exposition, in our notation, of the parts of Gauduchon's paper \cite{gau} relevant for us.

\begin{lemma}\label{levicom} For any vector fields $X,Y,Z$ on an almost Hermitian manifold $M$ we have $$(\levic_X \omega)(Y,Z) = \left \langle (\levic_X J)Y,Z \right \rangle.$$ \end{lemma}

\begin{proof} For vector fields $X,Y,Z \in \Gamma(TM)$ we observe
\begin{equation*}
    \begin{split} (\levic_X \omega) (Y \wedge Z)  &= X(\omega(Y \wedge Z)) - \omega( \levic_X(Y \wedge Z)) \\
        &= X\big(\langle JY, Z \rangle\big ) - \omega \big( (\levic_X Y) \wedge Z + Y \wedge (\levic_X Z)\big) \\
        &= \langle \levic_X JY, Z  \rangle + \langle JY, \levic_X Z \rangle - \langle J \levic_X Y, Z \rangle - \langle JY, \levic_X Z \rangle \\
        &= \langle \levic_X JY, Z  \rangle - \langle J \levic_X Y, Z \rangle = \left \langle (\levic_X J)Y,Z \right \rangle. \qedhere
    \end{split}
\end{equation*} \end{proof}

 In fact, the elements $N, d\omega$ and $ \levic\omega$ of $\Omega^2(TM)$ are related by the following well known fundamental identity of almost Hermitian geometry (see for instance \cite{Kob_Nom} where disparities in the conventional definitions of $N$ and the exterior derivative $d$ affect the coefficients in the statement)

\begin{lemma}\label{Kob_Nom}\cite{Kob_Nom} On any almost Hermitian manifold $M$ we have
$$ \tfrac{1}{2} \levic\omega(X,Y,Z) = \tfrac{1}{4} d\omega(X,Y,Z) - \tfrac{1}{4}d\omega(X,JY,JZ) + N(JX,Y,Z)  .\hspace{.5in} $$ \end{lemma}

\begin{definition} Let $\nabla$ be an affine metric connection.  We define the \textit{potential} $A^\nabla \in \Omega^2(TM)$ of $\nabla$ by $$A^\nabla(X,Y,Z) = \big \langle \nabla_X(Y) - \levic_X(Y) , Z \big \rangle .$$ Note that an affine connection $\nabla$ is metric if and only if $A^\nabla \in \Omega^2(TM)$. We write $A = A^\nabla$ when the referent connection is apparent. \end{definition}

\begin{proposition}\label{herm_iff1}\cite{gau} An affine metric connection $\nabla$ on $M$ is Hermitian if and only if $$ A(X,JY,Z) + A(X,Y,JZ) = - (\levic \omega)(X,Y,Z). $$ \end{proposition}

\begin{proof} By \Cref{levicom} we have for all vector fields $X,Y,Z$ that $$ (\levic \omega)(X,Y,Z) = \big \langle (\levic_X J)Y, Z \big \rangle  .$$ Using that $J$ is orthogonal and the definition of $A$ we observe $$ A(X,JY,Z) + A(X,Y,JZ) = \big \langle (\nabla_X J) Y , Z \big \rangle - \big \langle (\levic_X J) Y , Z \big  \rangle . $$ Consequently  $$ A(X,JY,Z) + A(X,Y,JZ) + (\levic \omega)(X,Y,Z) = 0 \textrm{ if and only if }  \big \langle (\nabla_X J) Y , Z \big \rangle = 0$$ for all vector fields $X,Y,Z$. That is, if and only if $(\nabla J) = 0 .$ \end{proof}

\begin{definition}\cite{gau} The subspaces $\Omega^{2,0}(TM),\ \Omega^{0,2}(TM),\ \Omega^{1,1}(TM)$ of $\Omega^2(TM)$ are defined by
\begin{equation*}
    \begin{split} \phi &\in \Omega^{2,0}(TM) \textrm{ if and only if } \phi(X,JY,Z) = -\phi(JX,Y,Z). \\
    \phi &\in \Omega^{0,2}(TM) \textrm{ if and only if } \phi(X,JY,Z) = \phantom{-}\phi(JX,Y,Z). \\
    \phi &\in \Omega^{1,1}(TM) \textrm{ if and only if } \phi(X,JY,JZ) = \phi(X,Y,Z).
    \end{split}
\end{equation*}  \end{definition}

 There are projections $$p_{j,k}: \Omega(TM) \rightarrow \Omega^{j,k}(TM) \textrm{ with }j,k \in \left \{0,1 \right \} \textrm{ and }j + k = 2 $$

defined for any $\phi \in \Omega(TM)$ by
\begin{equation*}
    \begin{split} p_{2,0}(\phi)(X,Y,Z)&=\tfrac{1}{4}\big(\phi(X,Y,Z) - \phi(X,JY,JZ) + \phi(JX,Y,JZ) + \phi(JX,JY,Z) \big), \\
        p_{0,2}(\phi)(X,Y,Z)&= \tfrac{1}{4}\big(\phi(X,Y,Z) - \phi(X,JY,JZ) - \phi(JX,Y,JZ) - \phi(JX,JY,Z) \big), \\
        p_{1,1}(\phi)(X,Y,Z) &= \tfrac{1}{2}\big(\phi(X,Y,Z) + \phi(X,JY,JZ) \big),
    \end{split}
\end{equation*}
so that the space $\Omega^2(TM)$ is endowed
with a natural orthogonal splitting $$\Omega^2(TM) = \Omega^{2,0}(TM) \oplus
\Omega^{1,1}(TM) \oplus \Omega^{0,2}(TM).$$ 
As an example, the Nijenhuis tensor $N \in \Omega^{0,2}(TM)$ as $N(X,JY,Z) = N(JX,Y,Z)$.

\begin{definition} For $\phi \in \Omega^2(TM)$ we set $p_{j,k}(\phi) = \phi^{j,k} .$ \end{definition}

\begin{lemma}\label{levic_om_11=0}\cite{gau} For any $\phi \in \Omega^{1,1}(TM)$ we have $\phi(X,JY,Z) + \phi(X,Y,JZ) = 0.$ Consequently $$(\levic \omega)^{1,1} = 0.$$ \end{lemma}

\begin{proof} We have $\phi(X,JY,Z) = \phi(X, J^2Y, JZ) = -\phi(X, Y, JZ)$ giving the first result. Writing $A = A^{2,0} + A^{1,1} + A^{0,2}$, \Cref{herm_iff1} implies \begin{equation*} 0 = A^{1,1}(X,JY,Z) + A^{1,1}(X,Y,JZ) = - (\levic \omega)^{1,1}(X,Y,Z). \qedhere \end{equation*} \end{proof}

The space $ \Omega^2(TM)$ also has two other distinguished projections which, apriori, do not involve the almost complex structure $J$.

\begin{definition}\cite{gau} We define the operator $\bp$ on $\Omega^2(TM)$ by $$\bp(\phi)(X,Y,Z)= \tfrac{1}{3}\big(\phi(X,Y,Z) + \phi(Y,Z,X) + \phi(Z,X,Y) \big) \hspace{.7in}$$
and the operator $\tr$, defined by $$\tr(\phi)(X,Y,Z)= \tfrac{1}{2n-1}\big(\sum_j\phi(v_j,v_j,Z)\langle X,Y \rangle - \sum_j\phi(v_j,v_j,Y)\langle Z, X \rangle \big)$$ where  $v_1, \dots, v_{2n}$ is a local orthonormal frame. \end{definition}

 Evidently $\bp$ is the identity on the restriction to $\Omega^3(M) \subset \Omega^2(TM)$ and so $\Omega^3(M)$ can be identified with the image of $\bp$ on $\Omega^2(TM)$. We also note that the map $\tr$ is actually the composition of the map onto $1$ forms, $r: \Omega^2(TM) \rightarrow \Omega^1(M)$ defined by $$r(\phi)(X) = \sum_j \phi(v_j,v_j,X)$$ with the map $i: \Omega^1(M) \rightarrow \Omega^2(TM)$ defined by $$i(\varphi)(X,Y,Z) = \tfrac{1}{2n-1}\big(\varphi(Z)\langle X,Y \rangle - \varphi(Y)\langle Z, X \rangle \big)$$ so that $i\circ r = \tr$. It is also easy to see that $r \circ i$ is the identity map on $\Omega^1(M)$, so that $i$ is an inclusion of $1$-forms into $\Omega(TM)$ and $r$ is a retraction of $\Omega^2(TM)$ onto $1$-forms. We conclude $\bp^2 = \bp $ and $\tr^2 = (ir)(ir) = i(ri)r = ir = \tr$ so that $\bp$ and $\tr$ are projections.

\begin{lemma} $\bp \tr = \tr \bp = 0.$ \end{lemma}

\begin{proof} Evidently $\tr \bp =0$ as $\tr$ vanishes on elements of $\Omega^3(M)$. That $\bp\tr = 0$ follows from
$$ \left (\bp \circ \tr \right)(\phi)(X,Y,Z) = \tr(\phi)(X,Y,Z) + \tr(\phi)(Y,Z,X) + \tr(\phi)(Z,X,Y)  $$ and applying the definition of $\tr(\phi)$. \end{proof}

Given two projections $\bp, \tr$ on a vector space $V$, satisfying $\bp \tr = \tr \bp = 0$ one can conclude $V = \textrm{Im}(\bp) \oplus \textrm{Im}(\tr) \oplus \big(\ker(\tr) \cap \ker(\bp)\big) $. Thus any $\phi \in \Omega(TM)$ can be uniquely written as $$\phi = \bp\phi + \tr\phi + \phi_0$$ where $\phi_0$ is the component of $\phi$ in $\ker(\bp)\cap\ker(\tr)$. Moreover we have the following identification

\begin{proposition}\label{decomp_forms}\cite{gau} $\Omega^2(TM) \cong \Omega^1(M) \oplus \Omega^3(M) \oplus \big(\ker(\tr) \cap \ker(\bp)\big).$ \end{proposition}

\begin{proof} This follows from the above remarks, along with the identifications $\Omega^1(M) \cong  \textrm{Im}(\tr)$ and $\Omega^3(M) \cong \textrm{Im}(\bp)$. \end{proof}

\begin{proposition}\label{A+T}\cite{gau} Let $T$ be the torsion of an affine metric connection and $A$ the potential. Then $A + T = 3 \bp(A) = \tfrac{3}{2}\bp(T). $ \end{proposition}

\begin{proof} We observe \small$$T(X,Y,Z) = \langle X, T(Y,Z) \rangle = \langle X, \nabla_Y Z - \nabla_Z Y - \levic_Y Z + \levic_Z Y \rangle =  A(Y,Z,X) + A(Z,X,Y) $$ \normalsize so that \small $$  A(X,Y,Z) + T(X,Y,Z)   = A(X,Y,Z) + A(Y,Z,X) + A(Z,X,Y) = 3\bp (A)(X,Y,Z). \hspace{.35in} $$ \normalsize Using again that $T(X,Y,Z) = A(Y,Z,X) + A(Z,X,Y)$  and the definition of $\bp$ one obtains $ 3 \bp(A) = \tfrac{3}{2}\bp(T) $. \end{proof}

Let $J$ be an almost complex structure on the complexified cotangent bundle $T_{\C}M^\vee$ and let $E_{\lambda}(\Jd)$ denote the $\lambda$ eigenspace of $\Jd$ where $\Jd $ is $J$ extended as a derivation on forms. We recall for $k = 0, \dots, 2n$ we have $$E_{ik}(\Jd) = \bigoplus_{k=p-q}\Omega^{p,q}(M).$$ We define the \textit{real} subspaces $E^+$ and $E^-$ of $\Omega^3(M)$ by  $$  E^{+} = \textrm{Re} \big( E_i(\Jd) \oplus E_{-i}(\Jd) \big) \ = \ \textrm{Re} \big( \Omega^{2,1}(M) \oplus \Omega^{1,2}(M) \big ),  $$
$$ E^{-} = \textrm{Re} \big( E_{3i}(\Jd) \oplus E_{-3i}(\Jd) \big) = \textrm{Re} \big( \Omega^{3,0}(M) \oplus \Omega^{0,3}(M) \big ) $$   so that the space of real 3-forms $\Omega^3(M) = E^{-} \oplus E^{+}.$  

\begin{definition} \cite{gau} For a 3-form $\varphi \in \Omega^3(M)$ we set $\varphi^+$ and $\varphi^-$ to be its components in $E^+$ and $E^-$ respectively. \end{definition}

 We note \cite{gau} that on the space $\Omega^3(M)$ the above decompositions are related as follows. For any 3-form $\varphi \in \Omega^3(M)$ one can check that $$  \varphi^- = \varphi^{0,2}, \hspace{.3in} $$
$$ \varphi^+ = \varphi^{2,0} + \varphi^{1,1} .$$ To illustrate this, if $\varphi \in E^-$ then $\Jd^2(\varphi) = -9\varphi$. Expanding $\Jd^2(\varphi)$ one gets $$ \Jd^2 \varphi(X,Y,Z) = -3 \varphi(X,Y,Z) + 2 \varphi(JX,JY,Z) + 2 \varphi(X,JY,JZ) + 2 \varphi(JX,Y,JZ) $$ and so $$ \varphi(X,Y,Z) = \tfrac{1}{4}\big( \varphi(X,Y,Z) -\varphi(X,JY,JZ) - \varphi(JX,Y,JZ) - \varphi(JX,JY,Z) \big) = \varphi^{0,2}(X,Y,Z) .$$

We now state and prove two useful lemmas involving the Nijenhuis tensor which we will require

\begin{lemma}\label{bpN=dom} $\bp N = \tfrac{1}{3}(d_c\omega)^-.$ \end{lemma}

\begin{proof} For $X \in T_{\C}(M)$ consider $g_X = \left ( X , \cdot \right )$ where $\left ( \cdot, \cdot \right )$ is the metric complex bilinearly extended to $T_\C(M)$. We observe that $N(X,Y,Z) $ (and hence $\bp N(X,Y,Z)$) is non-zero only if the vectors $X,Y,Z$ are all of pure type $(1,0)$ (or of type $(0,1)$) in the direct sum decomposition $T_{\C}(M) = T^{1,0}(M) \oplus T^{0,1}(M)$. We have
\begin{equation*}
    \begin{split}3\bp N(X,Y,Z) &= g_X(N(Y,Z)) + g_Y(N(Z,X)) + g_Z(N(X,Y)) \\
    &= N^\vee(g_X)(Y,Z) + N^\vee(g_Y)(Z,X) + N^\vee(g_Z)(X,Y) \\
    &= \big(\mu + \mub\big)(g_X)(Y,Z) + \big(\mu + \mub\big)(g_Y)(X,Z) + \big(\mu + \mub\big)(g_Z)(X,Y) \\
    &=-g_X\big([Y,Z]\big) - g_Y\big([Z,X]\big) - g_Z\big([X,Y]\big)
    \end{split}
\end{equation*}
where in the last equality we use that $d\varphi(X,Y) = X(\varphi(Y)) - Y(\varphi(X)) - \varphi[X,Y] $ for any $\varphi \in \Omega_\C^1(M)$, and that $\left ( X, Y \right ) = 0 $ for $X,Y$ both of type $(1,0)$ or type $(0,1)$. Furthermore on vectors $X,Y,Z$ all of type $(1,0)$ or $(0,1)$ we have 
\begin{equation*}
    \begin{split}
        d\omega(X,Y,Z) &= X(\omega(Y,Z)) -Y(\omega(X,Z))+ Z(\omega(X,Y)) \\
        &  - \omega([X,Y],Z) + \omega([X,Z],Y) -\omega([Y,Z],X) \hspace{1.1in}\\
        &=g_{JX}([Y,Z]) + g_{JY}([Z,X]) + g_{JZ}([X,Y]).
    \end{split}
\end{equation*} Noticing that $$d_c\omega(X,Y,Z) = - d\omega(JX,JY,JZ)$$ and again using that $X,Y,Z$ are all of pure type we obtain \begin{equation*} (d_c\omega)^-(X,Y,Z) = -g_X\big([Y,Z]\big) - g_Y\big([Z,X]\big) - g_Z\big([X,Y]\big) = 3\bp N(X,Y,Z). \hspace{.5in} \qedhere \end{equation*} \end{proof}

\begin{proposition}\label{02}\cite{gau} The map $\bp$ is invariant on $\Omega^{0,2}(TM)$. Furthermore the restriction of $\bp$ to $\Omega^{0,2}(TM)$ has image $E^-$. That is, the restriction is a surjective map $$\bp: \Omega^{0,2}(TM) \rightarrow E^-.$$ \end{proposition}

\begin{proof} We observe $\phi \in \Omega^{0,2}(TM)$ if and only if 
$$ \phi(X,Y,Z)= \phi^{0,2}(X,Y,Z) = \tfrac{1}{4}\big(\phi(X,Y,Z) - \phi(X,JY,JZ) - \phi(JX,Y,JZ) - \phi(JX,JY,Z) \big) $$ or equivalently
$$\phi(X,Y,Z) = -\tfrac{1}{3}\big( \phi(X,JY,JZ) + \phi(JX,Y,JZ) + \phi(JX,JY,Z) \big). $$ We compute
\begin{equation*}
    \begin{split} 3 \bp \phi (X,Y,Z) &= \phi(X,Y,Z) + \phi(Y,Z,X) + \phi(Z,X,Y) \\
    &= -\tfrac{1}{3} \Big( \phi(X,JY,JZ) + \phi(JX,Y,JZ) + \phi(JX,JY,Z)  \\
    &\hspace{.23in} +\phi(Y,JZ,JX) + \phi(JY,Z,JX) + \phi(JY,JZ,X) \\
    &\hspace{.23in} +\phi(Z,JX,JY) + \phi(JZ,X,JY) + \phi(JZ,JX,Y) \Big) \\
    &= -\bp\phi(X,JY,JZ) - \bp\phi(JX,Y,JZ) - \bp\phi(JX,JY,Z) \\
    \end{split}
\end{equation*} and thus $\bp \phi \in \Omega^{0,2}(TM)$. As $\bp \phi \in \Omega^3(M)$ as well we have $\bp \phi \in E^-$. Finally if $\varphi \in E^-$ then $\varphi = \varphi^{0,2}$ and  $\bp(\varphi^{0,2}) = \bp(\varphi) = \varphi$ so that the restriction is surjective. \end{proof}

\begin{definition}\cite{gau} We define the subspace $\Omega_s^{1,1}(TM) \subset \Omega^{1,1}(TM)$ by $$\Omega_s^{1,1}(TM)= \ker(\bp) \cap \Omega^{1,1}(TM).$$ Furthermore we define $\Omega_a^{1,1}(TM)$ to be its orthogonal complement in $\Omega^{1,1}(TM)$. \end{definition}

\begin{definition}\cite{gau} We define the operator $\M$ on $\Omega^2(TM)$ by $$\M(\phi)(X,Y,Z) = \phi(X,JY,JZ).$$ \end{definition}

 With this operator, the proof of the following two propositions can also be checked directly \cite{gau}.

\begin{proposition}\label{20}\cite{gau} The map $\bp$ restricted to $\Omega^{2,0}(TM)$ is an isomorphism $$\bp: \Omega^{2,0}(TM) \xrightarrow{\sim} E^+.$$ Furthermore for any $\phi \in \Omega^{2,0}(TM)$ we have $ \phi = \tfrac{3}{2}(\bp \phi - \M \bp \phi). $
\end{proposition}

\begin{proposition}\label{11}\cite{gau} The map $\bp$ restricted to $\Omega_a^{1,1}(TM)$ is an isomorphism $$\bp: \Omega_a^{1,1}(TM) \xrightarrow{\sim} E^+.$$ Furthermore for any $\phi \in \Omega_a^{1,1}(TM)$ we have $ \phi = \tfrac{3}{4}(\bp \phi + \M \bp \phi) .$ \end{proposition}

\begin{lemma}\label{bpphi^+} For any $\phi \in \Omega^2(TM)$ we have $\bp(\phi_a^{1,1}) + \bp(\phi^{2,0}) = (\bp \phi)^+$. \end{lemma}

\begin{proof} Evidently $\phi = \phi^{1,1} + \phi^{2,0} + \phi^{0,2}$ and so $$\bp\phi = \bp(\phi^{1,1}) + \bp(\phi^{2,0}) + \bp(\phi^{0,2}).$$ Thus, by Propositions \ref{20} and \ref{11} above we have \begin{equation*} (\bp\phi)^+ = \bp(\phi_a^{1,1})+\bp(\phi^{2,0}). \qedhere \end{equation*} \end{proof}

\begin{lemma}\label{bp_lev_om=dom} $\bp(\levic\omega) = \tfrac{1}{3}d\omega$. In particular $$\bp(\levic\omega^{2,0}) = \tfrac{1}{3}d\omega^+ \hspace{.1in} \textrm{ and } \hspace{.1in} \bp(\levic\omega^{0,2}) = \tfrac{1}{3}d\omega^- .\hspace{.1in} $$ \end{lemma}

\begin{proof} The identity $\bp(\levic \omega) = \tfrac{1}{3}d\omega$ follows at once from the observation that $$d\varphi(v_1, \dots, v_{k+1}) = \sum_j (-1)^{j+1} \levic_{v_j} \varphi(v_1, \dots, \widehat{v_j}, \dots, v_{k+1}) \textrm{ for any } k \textrm{ form } \varphi .$$ 
The remaining identities follows from the isomorphisms in Propositions \ref{20} and \ref{02} and the observation, proven in \Cref{levic_om_11=0}, that $(\levic \omega)^{1,1} = 0$. \end{proof}

\begin{lemma}\label{dom=3bpMdom} $(d\omega)^+ = 3\bp \M((d\omega)^+).$ \end{lemma}

\begin{proof} We have by \Cref{20} and \Cref{bp_lev_om=dom} that $$(\levic \omega)^{2,0} = \tfrac{3}{2}(\bp(\levic \omega^{2,0}) - \M(\bp( \levic\omega^{2,0})) = \tfrac{1}{2}((d\omega)^+ - \M(d\omega)^+). $$ Applying $\bp$ on both sides of the equality, and using \Cref{bp_lev_om=dom} once more, we conclude \begin{equation*} \tfrac{1}{3}(d\omega)^+ = \bp \M(d\omega^+). \qedhere \end{equation*} \end{proof}

\begin{lemma}\label{dom^-=bT^-} $ (d \omega)^-(X,Y,Z) = -3(\bp T)^-(JX,Y,Z). $ \end{lemma}

\begin{proof} Observe that by \Cref{herm_iff1} we have $$ - (\levic \omega)^{0,2}(X,Y,Z)= A^{0,2}(X,JY,Z) + A^{0,2}(X,Y,JZ) $$ so that 
\begin{equation*}
    \begin{split}  - 3 \bp\big((\levic \omega)^{0,2}\big)(X,Y,Z) &= A^{0,2}(X,JY,Z) + A^{0,2}(X,Y,JZ) \\
                                           &+ A^{0,2}(Y,JZ,X) +  A^{0,2}(Y,Z,JX) \\
                                           &+  A^{0,2}(Z,JX,Y) +  A^{0,2}(Z,X,JY).
    \end{split}
\end{equation*} Using the identity $T(X,Y,Z) = A(Y,Z,X) + A(Z,X,Y)$ obtained in the proof of \Cref{A+T}, we have 
\begin{equation*}
    \begin{split} - 3 \bp\big((\levic \omega)^{0,2}\big)(X,Y,Z) &= T^{0,2}(JY,Z,X) + T^{0,2}(JZ,X,Y) + T^{0,2}(JX,Y,Z) \\ 
                                                                &= T^{0,2}(Y,Z,JX) + T^{0,2}(Z,JX,Y) + T^{0,2}(JX,Y,Z)
    \end{split}
\end{equation*} where in the last equality we use that $\phi(JX,Y,Z) = \phi(X,JY,Z)$ for any $\phi \in \Omega^{0,2}(TM)$. We conclude $$- 3 \bp\big((\levic \omega)^{0,2}\big)(X,Y,Z) = 3 \bp \big(T^{0,2}\big)(JX,Y,Z) = 3 (\bp T)^-(JX,Y,Z)  $$ and the result follows by \Cref{bp_lev_om=dom}.

\end{proof}

\begin{lemma}\label{Jderid+} If $\varphi \in E^+$ then $\varphi(X,Y,Z) = \varphi(JX,JY,Z) + \varphi(X,JY,JZ) + \varphi(JX,Y,JZ).$ \end{lemma}

\begin{proof} For any $\varphi \in E^+$ we have $\Jd^2 \varphi= -\varphi$. But
$$ \Jd^2 \varphi(X,Y,Z) = -3 \varphi(X,Y,Z) + 2 \varphi(JX,JY,Z) + 2 \varphi(X,JY,JZ) + 2 \varphi(JX,Y,JZ) $$ which gives the result. \end{proof}

\begin{proposition}\label{herm_iff2}\cite{gau} An affine metric connection $\nabla$ on $M$ is Hermitian if and only if \begin{equation*}
    \begin{split} T^{2,0} - \tfrac{3}{2}\big(\bp (T^{1,1}) - \M\bp(T^{1,1})\big) &= \tfrac{1}{2}((d_c\omega)^+ - \M(d_c\omega)^+) \hspace{.1in} \textrm{ and } \\
        T^{0,2} &= N .
    \end{split}
\end{equation*} \end{proposition}

\begin{proof} By \Cref{herm_iff1} we have that $\nabla$ is hermitian if and only if 

$$A(X,JY,Z) + A(X,Y,JZ) = -(\levic\omega)(X,Y,Z).$$ Using \Cref{A+T} we rewrite the identity in \Cref{herm_iff1} as  \[ T(X,JY,Z) + T(X,Y,JZ) - \tfrac{3}{2}\big( \bp T(X,JY,Z) + \bp T(X,Y,JZ)\big ) = (\levic \omega)(X,Y,Z). \tag{$\alpha$} \label{eq:special0}  \]  
We decompose $T = T^{1,1} + T^{2,0} + T^{0,2}$ and collecting the $(0,2)$ parts of the expression (\ref{eq:special0}), by \Cref{02} we obtain \begin{equation}\label{eq:special1} 2T^{0,2}(JX,Y,Z) - 3(\bp T)^-(JX,Y,Z) = (\levic \omega)^{0,2}(X,Y,Z) . \end{equation}
We observe that by \Cref{dom^-=bT^-} we have $$ 2T^{0,2}(JX,Y,Z) + (d\omega)^-(X,Y,Z) = (\levic \omega)^{0,2}(X,Y,Z).  $$ By \Cref{Kob_Nom}, we see that $$2T^{0,2}(JX,Y,Z) = (\levic \omega)^{0,2}(X,Y,Z) - d\omega^-(X,Y,Z) = 2 N^{0,2}(JX,Y,Z) = 2 N(JX,Y,Z)$$ and we conclude equation \eqref{eq:special1} holds if and only if $ T^{0,2} = N$.

 We recall that by \Cref{levic_om_11=0} we have \[A^{1,1}(X,JY,Z) + A^{1,1}(X,Y,JZ) = -(\levic\omega)^{1,1}(X,Y,Z) =0 .  \] Collecting the $(1,1)$ parts of the expression (\ref{eq:special0}), we observe, as in the proof of \Cref{herm_iff1}, that $ T^{1,1}(X,JY,Z) + T^{1,1}(X,Y,JZ) = 0$ and so $$ \bp(T^{1,1})(X,JY,Z) + \bp(T^{1,1})(X,Y,JZ) =0. $$ Thus, by \Cref{bpphi^+} we have $$ (\bp T)^+(X,JY,Z) + (\bp T)^+(X,Y,JZ) = \bp(T^{2,0})(X,JY,Z) + \bp(T^{2,0})(X,Y,JZ). $$

 Finally, collecting the $(2,0)$ parts of the expression we have
\[ -2T^{2,0}(JX,Y,Z) -\tfrac{3}{2}\big( (\bp T)^+(X,JY,Z) + (\bp T)^+(X,Y,JZ) \big) = (\levic\omega)^{2,0}(X,Y,Z). \tag{$\gamma$} \label{eq:special2} \]
By \Cref{Jderid+} we rewrite \eqref{eq:special2} as
$$ -2T^{2,0}(JX,Y,Z) +\tfrac{3}{2}\big( (\bp T)^+(JX,Y,Z) - (\bp T)^+(JX,JY,JZ) \big) = (\levic\omega)^{2,0}(X,Y,Z). $$
Notice that by \Cref{bpphi^+} and \Cref{20} we have $(\bp T)^+ = \bp( T^{2,0}) + \bp (T^{1,1})$ and $T^{2,0} = \tfrac{3}{2}\big( \bp (T^{2,0}) - \M\bp (T^{2,0})\big)$. We rewrite \eqref{eq:special2} again as
$$ -T^{2,0}(JX,Y,Z) + \tfrac{3}{2} \big( \bp T^{1,1}(JX,Y,Z) - \M\bp T^{1,1}(JX,Y,Z)\big) = (\levic\omega)^{2,0}(X,Y,Z). $$ 
By \Cref{bp_lev_om=dom} and \Cref{20} we see $$ (\levic \omega)^{2,0} =\tfrac{3}{2} \big( \bp((\levic \omega)^{2,0}) - \M(\bp((\levic \omega)^{2,0})) \big) = \tfrac{1}{2}\big( (d\omega)^+ - \M(d\omega)^+\big) $$ so that \eqref{eq:special2} holds if and only if $$ T^{2,0} - \tfrac{3}{2}\big(\bp (T^{1,1}) - \M\bp(T^{1,1})\big) = \tfrac{1}{2}((d_c\omega)^+ - \M(d_c\omega)^+) .$$ Evidently $\nabla$ is hermitian if and only if \eqref{eq:special0}, which holds if and only if both \eqref{eq:special1} and \eqref{eq:special2}. \end{proof}

\begin{lemma}\label{bpTnew}\cite{gau} Let $\nabla$ be a Hermitian connection on $M$. Then $$\bp(T^{2,0} - T_a^{1,1}) = \tfrac{1}{3}d_c\omega^+.$$ \end{lemma}

\begin{proof} By the proof of \Cref{herm_iff2} we have $$T^{2,0} - \tfrac{3}{2}\big(\bp T^{1,1} - \M(\bp T^{1,1})\big) = (\nabla \omega)^{2,0}(J\cdot, \cdot, \cdot).$$ Applying $\bp$ to both sides of the equation we obtain by \Cref{bp_lev_om=dom} $$\bp(T^{2,0}) - \tfrac{3}{2}\big (\bp(T^{1,1}) - \bp \M \bp(T^{1,1}) \big) = \tfrac{1}{3}d_c\omega^+ .$$ Using that $T^{1,1}(X,JY,Z) + T^{1,1}(X,Y,JZ) = 0$ the identity $\bp(T^{1,1}) = 3 \bp \M \bp(T^{1,1})$ is readily verified. This gives the result. \end{proof}

Combining Lemmas \ref{bpphi^+} and \ref{bpTnew} we see that whenever $T$ is the torsion of a Hermitian connection on an almost Hermitian manifold $M$ we have the equations \begin{equation*}
    \begin{split}
        \bp(T_a^{1,1}) &= \tfrac{1}{2}( \bp T^+ - \tfrac{1}{3}(d_c\omega)^+ ),\\
        \bp(T^{2,0}) &= \tfrac{1}{2}( \bp T^+ + \tfrac{1}{3}(d_c\omega)^+).
    \end{split}
\end{equation*} This leads us to the following

\begin{theorem}\label{T=N+dometc}\cite{gau} On any almost Hermitian manifold $M$ the torsion of any Hermitian connection $\nabla$ on $M$ is given by
$$ T = N + \tfrac{9}{8}\bp T^+ + \tfrac{1}{8} d_c\omega^+ - \tfrac{3}{8}\M(\bp T^+) - \tfrac{3}{8} \M(d_c\omega^+) +T^{1,1}_s .$$ \end{theorem}

\begin{proof} Let $T$ be the torsion of a Hermitian connection. We have
\small
\begin{equation*}
    \begin{split}
        T &= T^{0,2} + T^{1,1} + T^{2,0} \\
        &= N + T_a^{1,1} + T^{2,0} + T_s^{1.1} \\
        &= N + \tfrac{3}{4}\big(\bp (T^{1,1}) + M \bp (T^{1,1})\big) + \tfrac{3}{2}\big(\bp (T^{1,1}) - \M\bp (T^{1,1})\big) + \tfrac{1}{2}(d_c\omega^+ - \M(d_c\omega^+)) + T_s^{1,1} \\
        &= N + \tfrac{3}{8}\big(3\bp T^+ - d_c\omega^+ - \M(\bp T^+) + \tfrac{1}{3}\M(d_c\omega^+) \big) + \tfrac{1}{2}\big(d_c\omega^+ -\M(d_c\omega^+)\big) + T_s^{1,1} \\
        &= N + \tfrac{9}{8}\bp T^+ + \tfrac{1}{8} d_c\omega^+ - \tfrac{3}{8}\M(\bp T^+) - \tfrac{3}{8} \M(d_c\omega^+) +T^{1,1}_s . \qedhere
    \end{split}
\end{equation*}
\normalsize \end{proof}

 Therefore we obtain the

\begin{corollary}\cite{gau} A Hermitian connection $\nabla$ is uniquely determined by choice of $T_s^{1,1}$ and $\bp T^+$. \end{corollary}

From the above Theorem and Corollary, Gauduchon \cite{gau} defines an affine line of `canonical' Hermitian connections on an almost Hermitian manifold. This is the set of Hermitian connections, $\nabla^t$, with torsion $T^t=T$ satisfying $$T_s^{1,1} = 0 \hspace{.1in}\textrm{ and } \hspace{.1in}\bp T^+ = \tfrac{2t-1}{3}(d_c\omega)^+ \textrm{ for any } t \in \mathbb{R}.$$  \Cref{T=N+dometc} then yields the following elegant characterization of any canonical Hermitian connection. Namely the torsion of $\nabla^t$ is given by
\[ T^t = N + \tfrac{3t-1}{4}d_c\omega^+ - \tfrac{t+1}{4}\M(d_c\omega^+). \tag{CT} \label{eq:special3} \]
For example, a natural choice of Hermitian connection is the \textit{Chern connection} $\nabla^{\textrm{ch}}$ defined by requiring $T^{1,1}=0$ or equivalently $t=1$. It is readily seen that, in the case of a K{\"a}hler manifold, we have the well known identity $\nabla^{\textrm{ch}} = \levic $ as all torsion components of $\nabla^{\textrm{ch}}$ vanish. 

What is more, evidently all of the canonical Hermitian connections agree in the case that $d \omega = 0$. We denote the potential of a canonical hermitian connection by $A^t$. By \Cref{A+T} we observe $$A^t = -T^t + \tfrac{3}{2} \bp T^t .$$ We then have

\begin{proposition}\label{A^t=N+dometc}\cite{gau} Let $M$ be an almost Hermitian manifold and let $\nabla^t$ be a canonical Hermitian connection in the sense of Gauduchon. Then $$A^t = -N + \tfrac{3}{2}\bp N + \tfrac{t-1}{4}d_c\omega^+ + \tfrac{t+1}{4}\M(d_c\omega^+).$$ \end{proposition}

\begin{proof} The result follows by substituting the identity \eqref{eq:special3} in $A^t = -T^t + \tfrac{3}{2} \bp T^t $. We observe that  \[ \bp T^t = \bp N + \tfrac{3t-1}{4}d_c\omega^+ - \tfrac{t+1}{4} \bp \M(d_c \omega^+) = \bp N + \tfrac{2t-1}{3} d_c\omega^+,\] where we use that $\bp \M(d_c \omega^+) = \tfrac{1}{3}d_c\omega^+$ which follows by \Cref{dom=3bpMdom}. \end{proof}

\end{document}